\documentclass[a4paper,11pt]{article}

\usepackage[hiresbb]{graphicx}
\usepackage{amsmath,amsthm,amssymb}
\usepackage{natbib,color,geometry}

\usepackage{hyperref}
\hypersetup{%
 colorlinks = true,
 allcolors = blue,
 urlcolor =  black}

\makeatletter
\let\@fnsymbol\@arabic
\makeatother

\theoremstyle{plain} 
\newtheorem{thm}{Theorem}[section]%
\newtheorem{prop}[thm]{Proposition}%
\newtheorem{lem}[thm]{Lemma}%
\newtheorem{cor}[thm]{Corollary}%
\newtheorem{exam}{Example}%
\newtheorem{rem}{Remark}%

\numberwithin{equation}{section}

\def\ve{ \varepsilon }
\def\rk{ {\mathrm{rank}} }
\def\E{ {\mathrm{E}}}
\def\V{ {\mathrm{Cov}}}
\def\MC{\mathcal{C}}

\title{\textbf{Equality between two general ridge estimators and  equivalence of their residual sums of squares}}
\author{Hirai Mukasa\footnote{Joint Graduate School of Mathematics for Innovation, Kyushu University, 744, Motooka, Nishi-ku, Fukuoka 819-0395, Japan.
mail: \url{mukasa.hirai.075@s.kyushu-u.ac.jp}} \ and  Koji Tsukuda\footnote{Faculty of Mathematics, Kyushu University, 744, Motooka, Nishi-ku, Fukuoka 819-0395, Japan. mail: \url{tsukuda@math.kyushu-u.ac.jp}}}
\date{}

\begin{document}
\maketitle

\begin{abstract}
General ridge estimators are typical linear estimators in a general linear model.
The class of them includes some shrinkage estimators in addition to classical linear unbiased estimators such as the ordinary least squares estimator and the weighted least squares estimator.
We derive necessary and sufficient conditions under which two general ridge estimators coincide. 
In particular, two noteworthy conditions are added to those from previous studies.
The first condition is given as a seemingly column space relationship to the covariance matrix of the error term, and the second one is based on the biases of general ridge estimators.
Another problem studied in this paper is to derive an equivalence condition such that equality between two residual sums of squares holds when general ridge estimators are considered.
\\
\\
\textbf{Keywords}: General linear model, Least squares estimator, Residual sum of squares, Ridge estimator\\
\textbf{Mathematics Subject Classification}: 62J05, 62F10, 62J07
\end{abstract}

\section{Introduction}\label{sec:1}	
For positive integers $m_1$ and $m_2$, let $\mathcal{M}(m_1,m_2)$ denote the set of $m_1 \times m_2$ matrices. 
For a positive integer $m$, let $\mathcal{S}^+(m)$ and $\mathcal{S}^N(m)$ denote the sets of $m \times m$ positive definite and positive semidefinite matrices, respectively. 
Furthermore, let $\boldsymbol{I}_m \in \mathcal{S}^+(m)$ denote the $m\times m$ identity matrix. The transpose, rank, and column space of a matrix $\boldsymbol{M}$ are denoted by $\boldsymbol{M}^\top$, $\rk(\boldsymbol{M})$, and $\MC(\boldsymbol{M})$, respectively. 
For matrices $\boldsymbol{P}$ and $\boldsymbol{Q}$ with the same number of rows, $(\boldsymbol{P} \quad \boldsymbol{Q})$ denotes the partitioned matrix with $\boldsymbol{P}$ and $\boldsymbol{Q}$ as submatrices. 
The expectation vector and covariance matrix of an $m \times 1$ random variable vector $\boldsymbol{u}$ are denoted by $\E[\boldsymbol{u}]$ and $\V(\boldsymbol{u})$, respectively.  

Let $n$ and $k$ be positive integers satisfying $n > k$. 
Throughout this paper, for an observable random vector $\boldsymbol{y} \in \mathbb{R}^n$, let us consider the general linear model
\begin{equation}\label{glm}
 \boldsymbol{y} = \boldsymbol{X} \boldsymbol{\beta} +\boldsymbol{\ve}, \quad \E[\boldsymbol{\ve}] = \boldsymbol{0}, \quad \V(\boldsymbol{\ve}) = \sigma^2 \boldsymbol{\Omega},  
\end{equation}
where  $\boldsymbol{X} \in \mathcal{M}(n,k)$ is a known non-random matrix with $\rk(\boldsymbol{X}) = k$, $\boldsymbol{\beta} \in \mathbb{R}^k$ is an unknown vector, $\boldsymbol{\ve} \in \mathbb{R}^n$ is an unobservable random vector, $\sigma^2 \in (0, \infty)$ is an unknown constant, and $\boldsymbol{\Omega} \in \mathcal{S}^+(n)$ is a known positive definite matrix. 
A matrix $\boldsymbol{Z} \in \mathcal{M}(n,n-k)$ satisfying $\boldsymbol{X}^\top \boldsymbol{Z} = \boldsymbol{0}$ and $\rk(\boldsymbol{Z}) = n-k$ will be fixed throughout the paper. 
To simplify the notation, let us denote $\boldsymbol{N}  = \boldsymbol{I}_n - \boldsymbol{X} (\boldsymbol{X}^\top \boldsymbol{X})^{-1} \boldsymbol{X}^\top \left(= \boldsymbol{Z} (\boldsymbol{Z}^\top \boldsymbol{Z})^{-1} \boldsymbol{Z}^\top\right)$. 
In this paper, given two general ridge estimators for $\boldsymbol{\beta}$ in \eqref{glm}, we first discuss necessary and sufficient conditions for their equality. 
We also discuss a condition on $\boldsymbol{\Omega}$ under which an equality between two (generalized) residual sums of squares holds.

When linear unbiased estimation of $\boldsymbol{\beta}$ in the general linear model \eqref{glm} is considered, it is sufficient to discuss the class of weighted least squares estimators of the form
\[
\hat{\boldsymbol{\beta}}(\boldsymbol{\Phi},\boldsymbol{0}) =  (\boldsymbol{X}^\top \boldsymbol{\Phi}^{-1} \boldsymbol{X})^{-1} \boldsymbol{X}^\top \boldsymbol{\Phi}^{-1} \boldsymbol{y}, \quad \boldsymbol{\Phi} \in \mathcal{S}^+(n).
\]
Let $\mathcal{D} = \{ \hat{\boldsymbol{\beta}}(\boldsymbol{\Phi}, \boldsymbol{0});  \boldsymbol{\Phi} \in \mathcal{S}^+(n)\}$ denote the class of weighted least squares estimators. 
Then, $\hat{\boldsymbol{\beta}} (\boldsymbol{\Omega}, \boldsymbol{0})$ is optimal in $\mathcal{D}$ in terms of L\"{o}wner partial ordering with respect to the matrix valued risk function
\begin{align*}
\mathbb{R}^k \times \mathcal{D} \ni (\boldsymbol{\beta}, \delta) \mapsto \E\left[ (\boldsymbol{\beta} - \delta) (\boldsymbol{\beta} - \delta)^\top\right].
\end{align*}  
In other words, $\hat{\boldsymbol{\beta}} (\boldsymbol{\Omega}, \boldsymbol{0})$ is the best linear unbiased estimator (BLUE). In the class $\mathcal{D}$, the efficiency of the ordinary least square estimator $\hat{\boldsymbol{\beta}}(\boldsymbol{I}_n, \boldsymbol{0})$ is a classical problem, and the optimality of $\hat{\boldsymbol{\beta}} (\boldsymbol{I}_n, \boldsymbol{0})$ has been intensively discussed. 
\cite{RefA48} mentioned that a condition under which $\hat{\boldsymbol{\beta}}(\boldsymbol{I}_n, \boldsymbol{0})$ is identically equal to $\hat{\boldsymbol{\beta}} (\boldsymbol{\Omega}, \boldsymbol{0})$, that is, the equality
\begin{equation}\label{OBE}
 \hat{\boldsymbol{\beta}}(\boldsymbol{\Omega}, \boldsymbol{0}) = \hat{\boldsymbol{\beta}} (\boldsymbol{I}_n, \boldsymbol{0}) \quad \mbox{for any} \ \ \boldsymbol{y} \in \mathbb{R}^n
\end{equation}
holds, is that the column vector of $\boldsymbol{X}$ can be expressed as a linear combination of eigenvectors of $\boldsymbol{\Omega}$.
Since then, conditions for \eqref{OBE} to hold have been thoroughly investigated under general settings in many studies, including \cite{RefR67}, \cite{RefW67}, \cite{RefZ67}, and \cite{RefK68}.
For more details and other studies, see, for example, \cite{RefPS89} and \cite{RefKK04}. 
\cite{RefR67} proved that \eqref{OBE} holds if and only if $\boldsymbol{\Omega}$ is of the form
\begin{equation}\label{Rcond}
\boldsymbol{\Omega} = \boldsymbol{X} \boldsymbol{\Gamma} \boldsymbol{X}^\top + \boldsymbol{Z} \boldsymbol{\Delta} \boldsymbol{Z}^\top
\quad 
\mbox{for some} \ \ \boldsymbol{\Gamma} \in \mathcal{S}^+(k), \  \boldsymbol{\Delta} \in \mathcal{S}^+(n-k) ,
\end{equation}
where \eqref{Rcond} is the form mentioned by \cite{RefG70}.
Additionally, \cite{RefK68} showed that \eqref{OBE} is equivalent to a column space relationship
\begin{equation}\label{Kcond1}
\MC(\boldsymbol{X}) = \MC(\boldsymbol{\Omega} \boldsymbol{X}),
\end{equation}
which is also equivalent to
\begin{equation}\label{Kcond2}
\boldsymbol{X} = \boldsymbol{\Omega} \boldsymbol{X} \boldsymbol{G}
\quad 
\mbox{for some} \ \ \boldsymbol{G} \in \mathcal{M}(k,k),
\end{equation}
because $\boldsymbol{\Omega}$ is nonsingular.

If we consider linear estimators that are not necessarily unbiased, general ridge estimators of the form
\[
\hat{\boldsymbol{\beta}}(\boldsymbol{\Phi},\boldsymbol{K}) =  (\boldsymbol{X}^\top \boldsymbol{\Phi}^{-1} \boldsymbol{X} + \boldsymbol{K})^{-1} \boldsymbol{X}^\top \boldsymbol{\Phi}^{-1} \boldsymbol{y}, \quad 
\boldsymbol{\Phi} \in \mathcal{S}^+(n), \ \boldsymbol{K} \in \mathcal{S}^N(k)
\]
play a special role.
In this paper, we consider that $\boldsymbol{\Phi} \in \{ \boldsymbol{I}_n, \boldsymbol{\Omega} \}$.
For example,  $\boldsymbol{K} = \boldsymbol{0}$ corresponds to a weighted least squares estimator, $\boldsymbol{K} = \lambda \boldsymbol{I}_k$ to an ordinary ridge estimator $(\lambda>0)$, and $\boldsymbol{K} = \rho \boldsymbol{X}^\top \boldsymbol{\Phi}^{-1} \boldsymbol{X}$ to a typical shrinkage estimator $(\rho > 0)$ in $\hat{\boldsymbol{\beta}}(\boldsymbol{\Phi},\boldsymbol{K})$.
For more details, see \cite{RefHK70} and \cite{RefG03}. 
When $\boldsymbol{K}$ is positive definite, the so-called Kuks--Olman estimator $\hat{\boldsymbol{\beta}}(\boldsymbol{\Omega}, \boldsymbol{K})$ is derived as an optimal estimator from a minimax principle under the squared risk function in certain special cases.
In the theory of linear inference, properties such as linear sufficiency and admissibility in the class of linear estimators are of interest; see, for example, \cite{RefC66}, \cite{RefR76}, and \cite{RefD83}.
As shown by \cite{RefR76} and \cite{RefM96}, the general ridge estimator $\hat{\boldsymbol{\beta}}(\boldsymbol{\Omega},\boldsymbol{K})$ is linearly sufficient and admissible within the class of linear estimators under the squared loss function, and linear estimators satisfying these properties simultaneously are also general ridge estimators.
For properties of general ridge estimators, see, for example, \cite{RefL75}, \cite{RefB75}, \cite{RefR76}, \cite{RefM96}, \cite{RefG98}, \cite{RefAS00,RefAS11}, and \cite{RefGM04}. 
Hence, it is fundamental to consider general ridge estimators as linear estimators, and we wish to extend \eqref{OBE} for given $\boldsymbol{K}_1, \boldsymbol{K}_2 \in \mathcal{S}^N(k)$ by deriving necessary and sufficient conditions for the equality
\begin{equation}\label{R2E}
 \hat{\boldsymbol{\beta}}(\boldsymbol{\Omega}, \boldsymbol{K}_1) = \hat{\boldsymbol{\beta}} (\boldsymbol{I}_n, \boldsymbol{K}_2) \quad \mbox{for any} \ \ \boldsymbol{y} \in \mathbb{R}^n
\end{equation}
to hold. Such conditions are also of interest because when \eqref{R2E} holds, properties of $\hat{\boldsymbol{\beta}}(\boldsymbol{\Omega}, \boldsymbol{K}_1)$ such as linear sufficiency, linear admissibility, and minimaxity in some situations are shared by $\hat{\boldsymbol{\beta}} (\boldsymbol{I}_n, \boldsymbol{K}_2)$.
\cite{RefTK20} derived an equivalence condition for \eqref{R2E} that extends \eqref{Rcond}. 
In Section~\ref{sec:3}, we discuss necessary and sufficient conditions for \eqref{R2E} in more detail and give two new conditions.
The first one is a condition corresponding to \eqref{Kcond2} and the second one is a condition based on the biases of $\hat{\boldsymbol{\beta}}(\boldsymbol{\Omega}, \boldsymbol{K}_1)$ and $\hat{\boldsymbol{\beta}}(\boldsymbol{I}_n, \boldsymbol{K}_2)$. 
In addition, we derive necessary conditions for \eqref{R2E} under the specific assumption that $\boldsymbol{K}_1$ and $\boldsymbol{K}_2$ are idempotent matrices.

Furthermore, if we consider a quadratic estimator of $\sigma^2$, particularly when assuming a normal distribution of the observations, we often use the (generalized) residual sums of squares of the estimator $\hat{\boldsymbol{\beta}}(\boldsymbol{\Phi},\boldsymbol{0})$
\begin{align*}
RSS(\boldsymbol{\Phi},\boldsymbol{0}) &= \left\| \boldsymbol{\Phi}^{-1/2} \left( \boldsymbol{y} - \boldsymbol{X} \hat{\boldsymbol{\beta}}(\boldsymbol{\Phi},\boldsymbol{0}) \right) \right\|^2\\
&=\left(\boldsymbol{y} - \boldsymbol{X} \hat{\boldsymbol{\beta}}(\boldsymbol{\Phi},\boldsymbol{0})\right)^\top \boldsymbol{\Phi}^{-1} \left(\boldsymbol{y} - \boldsymbol{X} \hat{\boldsymbol{\beta}}(\boldsymbol{\Phi},\boldsymbol{0})\right), \quad
\boldsymbol{\Phi} \in \mathcal{S}^+(n),    
\end{align*}
where $\| \cdot \|$ stands for the Euclidean norm.
Then, necessary and sufficient condition for the two equalities $\hat{\boldsymbol{\beta}}(\boldsymbol{\Omega},\boldsymbol{0}) = \hat{\boldsymbol{\beta}} (\boldsymbol{I}_n,\boldsymbol{0})$ and $RSS(\boldsymbol{\Omega},\boldsymbol{0}) = RSS(\boldsymbol{I}_n,\boldsymbol{0})$ to hold simultaneously for any $\boldsymbol{y} \in \mathbb{R}^n$ are also of interest. 
In fact, \cite{RefK68} and \cite{RefK80} dealt with this topic.
Moreover, in terms of the likelihood ratio test for the sphericity of the covariance matrix $\V(\boldsymbol{y}) = \sigma^2 \boldsymbol{\Omega}$ of the observations, the condition for the equality
\begin{equation}\label{RSE}
RSS(\boldsymbol{\Omega},\boldsymbol{0}) = RSS(\boldsymbol{I}_n,\boldsymbol{0})  \quad \mbox{for any} \ \ \boldsymbol{y} \in \mathbb{R}^n 
\end{equation}
to hold is of interest. 
\cite{RefK80} proved that \eqref{RSE} holds if and only if $\boldsymbol{\Omega}$ is of the form
\begin{equation}\label{Kcond}
\boldsymbol{\Omega} = \boldsymbol{N} + \boldsymbol{\Lambda}  - \boldsymbol{N} \boldsymbol{\Lambda} \boldsymbol{N}
\quad 
\mbox{for some symmetric} \ \  \boldsymbol{\Lambda} \in \mathcal{M}(n,n)
\end{equation}
from more general results.
In Section~\ref{sec:2}, we rewrite the condition \eqref{Kcond} into an equivalent form to extend it.
Hereafter, we denote the residual sum of squares under the general ridge estimator $\hat{{\boldsymbol{\beta}}}(\boldsymbol{\Phi},\boldsymbol{K})$ by
\[ RSS(\boldsymbol{\Phi},\boldsymbol{K}) = \left\| \boldsymbol{\Phi}^{-1/2} \left( \boldsymbol{y} - \boldsymbol{X} \hat{\boldsymbol{\beta}}(\boldsymbol{\Phi}, \boldsymbol{K}) \right) \right\|^2, \quad
\boldsymbol{\Phi} \in \mathcal{S}^+(n), \ 
\boldsymbol{K} \in \mathcal{S}^N(k).  \] 
The residual sum of squares $ RSS(\boldsymbol{\Phi},\boldsymbol{K})$ is often computed on data analysis even when $\boldsymbol{K} \neq \boldsymbol{0}$; see, for example, \cite{RefHK70}. 
In Section~\ref{sec:4}, for given $\boldsymbol{K}_1, \boldsymbol{K}_2 \in \mathcal{S}^N(k)$, we derive a necessary and sufficient condition for the equality
\begin{equation}\label{GRSE}
RSS(\boldsymbol{\Omega},\boldsymbol{K}_1) = RSS(\boldsymbol{I}_n, \boldsymbol{K}_2) \quad \mbox{for any} \ \ \boldsymbol{y} \in \mathbb{R}^n 
\end{equation}
to hold as a generalization of \eqref{RSE}.

\section{Preliminaries}\label{sec:2}
In this section, we make preparations for some results that are used in subsequent sections.

First, we introduce a useful representation of $\boldsymbol{\Omega}$ and $\boldsymbol{\Omega}^{-1}$; for these results, see \cite{RefK98}.

\begin{lem}\label{OOI}
It holds that
\begin{equation}\label{Orep}
 \boldsymbol{\Omega} 
= \left( \begin{matrix} \boldsymbol{X} & \boldsymbol{Z} \end{matrix}\right) \left(\begin{matrix} \boldsymbol{\Gamma} & \boldsymbol{\Xi} \\ \boldsymbol{\Xi}^\top & \boldsymbol{\Delta} \end{matrix} \right) \left( \begin{matrix} \boldsymbol{X}^\top \\ \boldsymbol{Z}^\top \end{matrix}\right)
\end{equation}
for some  $\boldsymbol{\Gamma} \in \mathcal{S}^+(k)$, $\boldsymbol{\Xi} \in \mathcal{M}(k,n-k)$, $\boldsymbol{\Delta} \in \mathcal{S}^+(n-k)$.
Moreover, when $\boldsymbol{\Omega}$ is expressed as \eqref{Orep}, it holds that
\[
\boldsymbol{\Omega}^{-1} 
= \left( \begin{matrix} \boldsymbol{X}(\boldsymbol{X}^\top \boldsymbol{X})^{-1} & \boldsymbol{Z}(\boldsymbol{Z}^\top \boldsymbol{Z})^{-1} \end{matrix}\right) \left(\begin{matrix} \boldsymbol{A} & \boldsymbol{B} \\ \boldsymbol{C} & \boldsymbol{D}\end{matrix} \right) \left( \begin{matrix} (\boldsymbol{X}^\top \boldsymbol{X})^{-1} \boldsymbol{X}^\top \\ (\boldsymbol{Z}^\top \boldsymbol{Z})^{-1} \boldsymbol{Z}^\top \end{matrix}\right),
\]
where
\begin{equation} \label{ABCD}
\boldsymbol{A} = (\boldsymbol{\Gamma} - \boldsymbol{\Xi} \boldsymbol{\Delta}^{-1} \boldsymbol{\Xi}^\top)^{-1}, \ 
\boldsymbol{B} = -\boldsymbol{A}\boldsymbol{\Xi} \boldsymbol{\Delta}^{-1}, \
\boldsymbol{C} = \boldsymbol{B}^\top, \
\boldsymbol{D} = \boldsymbol{\Delta}^{-1} + \boldsymbol{B}^\top \boldsymbol{A}^{-1} \boldsymbol{B}.
\end{equation}
\end{lem}

\begin{rem}
The representation \eqref{Orep} holds when $\boldsymbol{\Gamma}, \boldsymbol{\Xi}, \boldsymbol{\Delta}$ are chosen as follows:
\begin{align*}
&\boldsymbol{\Gamma} = (\boldsymbol{X}^\top \boldsymbol{X})^{-1} \boldsymbol{X}^\top \boldsymbol{\Omega} \boldsymbol{X} (\boldsymbol{X}^\top \boldsymbol{X})^{-1}, \ 
\boldsymbol{\Xi} =  (\boldsymbol{X}^\top \boldsymbol{X})^{-1} \boldsymbol{X}^\top \boldsymbol{\Omega} \boldsymbol{Z} (\boldsymbol{Z}^\top \boldsymbol{Z})^{-1},\\
& \boldsymbol{\Delta} =  (\boldsymbol{Z}^\top \boldsymbol{Z})^{-1} \boldsymbol{Z}^\top \boldsymbol{\Omega} \boldsymbol{Z} (\boldsymbol{Z}^\top \boldsymbol{Z})^{-1}.
\end{align*}
\end{rem}

\begin{rem}
The equality \eqref{Rcond} is equivalent to $\boldsymbol{\Xi} = \boldsymbol{0}$ in \eqref{Orep} and also to $\boldsymbol{X}^\top \boldsymbol{\Omega} \boldsymbol{Z} = \boldsymbol{0}$ and to $\boldsymbol{X}^\top \boldsymbol{\Omega}^{-1} \boldsymbol{Z} = \boldsymbol{0}$.
\end{rem}

The following Proposition~\ref{prop1} is an equivalent rewriting of the condition \eqref{Kcond} of \cite{RefK80} using Lemma~\ref{OOI}, which will be used later.

\begin{prop}\label{prop1}
The equality \eqref{RSE} holds if and only if $\boldsymbol{\Delta} = (\boldsymbol{Z}^\top \boldsymbol{Z})^{-1}$ in \eqref{Orep}.
\end{prop}

\begin{proof}
For simplicity, the notation in \eqref{ABCD} is used. 
Let us denote
\begin{align*}
\boldsymbol{L}=\boldsymbol{I}_n-\boldsymbol{X}\boldsymbol{A}^{-1}\boldsymbol{X}^{\top}\boldsymbol{\Omega}^{-1}.    
\end{align*}
Since
\begin{align*}
\boldsymbol{y}-\boldsymbol{X}\hat{\boldsymbol{\beta}}(\boldsymbol{\Omega},\boldsymbol{0}) &=\boldsymbol{y}-\boldsymbol{X}(\boldsymbol{X}^{\top}\boldsymbol{\Omega}^{-1}\boldsymbol{X})^{-1}\boldsymbol{X}^{\top}\boldsymbol{\Omega}^{-1}\boldsymbol{y}\\
&=\left\{\boldsymbol{I}_n-\boldsymbol{X}(\boldsymbol{X}^{\top}\boldsymbol{\Omega}^{-1}\boldsymbol{X})^{-1}\boldsymbol{X}^{\top}\boldsymbol{\Omega}^{-1}\right\}\boldsymbol{y}\\
&=(\boldsymbol{I}_n-\boldsymbol{X}\boldsymbol{A}^{-1}\boldsymbol{X}^{\top}\boldsymbol{\Omega}^{-1})\boldsymbol{y}\\
&=\boldsymbol{L}\boldsymbol{y}
\end{align*}
and
\begin{align*}
\boldsymbol{y}-\boldsymbol{X}\hat{\boldsymbol{\beta}}(\boldsymbol{I}_n,\boldsymbol{0}) &=\boldsymbol{y}-\boldsymbol{X}(\boldsymbol{X}^{\top}\boldsymbol{X})^{-1}\boldsymbol{X}^{\top}\boldsymbol{y}\\
&=\left\{\boldsymbol{I}_n-\boldsymbol{X}(\boldsymbol{X}^{\top}\boldsymbol{X})^{-1}\boldsymbol{X}^{\top}\right\}\boldsymbol{y}\\
&=\boldsymbol{N}\boldsymbol{y} ,
\end{align*}
the residual sums of squares $RSS(\boldsymbol{\Omega},\boldsymbol{0})$ and $RSS(\boldsymbol{I}_n,\boldsymbol{0})$ can be rewritten as
\begin{align*}
&RSS(\boldsymbol{\Omega},\boldsymbol{0})=\left(\boldsymbol{y}-\boldsymbol{X}\hat{\boldsymbol{\beta}}(\boldsymbol{\Omega},\boldsymbol{0})\right)^{\top}\boldsymbol{\Omega}^{-1}\left(\boldsymbol{y}-\boldsymbol{X}\hat{\boldsymbol{\beta}}(\boldsymbol{\Omega},\boldsymbol{0})\right)=\boldsymbol{y}^\top  \boldsymbol{L}^\top \boldsymbol{\Omega}^{-1} \boldsymbol{L} \boldsymbol{y}
\end{align*}
and
\begin{align*}
&RSS(\boldsymbol{I}_n,\boldsymbol{0})=\left(\boldsymbol{y}-\boldsymbol{X}\hat{\boldsymbol{\beta}}(\boldsymbol{I}_n,\boldsymbol{0})\right)^{\top}\left(\boldsymbol{y}-\boldsymbol{X}\hat{\boldsymbol{\beta}}(\boldsymbol{I}_n,\boldsymbol{0})\right)=\boldsymbol{y}^\top \boldsymbol{N}\boldsymbol{N} \boldsymbol{y} ,
\end{align*}
respectively.
Thus, the problem is to derive a condition under which $\boldsymbol{y}^\top ( \boldsymbol{L}^\top \boldsymbol{\Omega}^{-1} \boldsymbol{L}-\boldsymbol{N}\boldsymbol{N} )\boldsymbol{y}=0$ holds for any $\boldsymbol{y}\in \mathbb{R}^{n}$, which is equivalent to deriving a condition for 
\begin{align}\label{LN}
\boldsymbol{L}^\top \boldsymbol{\Omega}^{-1} \boldsymbol{L}-\boldsymbol{N}\boldsymbol{N} =\boldsymbol{0}.    
\end{align}
Since the matrix $(\boldsymbol{X} \quad \boldsymbol{Z})$ is nonsingular, the equality \eqref{LN} holds if and only if the following three equalities simultaneously hold:
\begin{align}
&\boldsymbol{X}^\top (\boldsymbol{L}^\top \boldsymbol{\Omega}^{-1} \boldsymbol{L}-\boldsymbol{N}\boldsymbol{N}) \boldsymbol{X} =\boldsymbol{0} , \label{XX} \\
&\boldsymbol{X}^\top (\boldsymbol{L}^\top \boldsymbol{\Omega}^{-1} \boldsymbol{L}-\boldsymbol{N}\boldsymbol{N}) \boldsymbol{Z} =\boldsymbol{0} , \label{XZ} \\
&
\boldsymbol{Z}^\top (\boldsymbol{L}^\top \boldsymbol{\Omega}^{-1} \boldsymbol{L}-\boldsymbol{N}\boldsymbol{N}) \boldsymbol{Z} =\boldsymbol{0} . \label{ZZ}
\end{align}
The quantities $\boldsymbol{L}^\top \boldsymbol{\Omega}^{-1} \boldsymbol{L}$ and $\boldsymbol{N}\boldsymbol{N}$ are calculated to be
 \begin{align*}
\boldsymbol{L}^\top \boldsymbol{\Omega}^{-1} \boldsymbol{L} &= (\boldsymbol{I}_n-\boldsymbol{\Omega}^{-1}\boldsymbol{X}\boldsymbol{A}^{-1}\boldsymbol{X}^{\top})\boldsymbol{\Omega}^{-1}(\boldsymbol{I}_n-\boldsymbol{X}\boldsymbol{A}^{-1}\boldsymbol{X}^{\top}\boldsymbol{\Omega}^{-1})\\
&=\boldsymbol{\Omega}^{-1}-2\boldsymbol{\Omega}^{-1}\boldsymbol{X}\boldsymbol{A}^{-1}\boldsymbol{X}^{\top}\boldsymbol{\Omega}^{-1}+\boldsymbol{\Omega}^{-1}\boldsymbol{X}\boldsymbol{A}^{-1}\boldsymbol{X}^{\top}\boldsymbol{\Omega}^{-1}\boldsymbol{X}\boldsymbol{A}^{-1}\boldsymbol{X}^{\top}\boldsymbol{\Omega}^{-1}\\
&=\boldsymbol{\Omega}^{-1}-2\boldsymbol{\Omega}^{-1}\boldsymbol{X}\boldsymbol{A}^{-1}\boldsymbol{X}^{\top}\boldsymbol{\Omega}^{-1}+\boldsymbol{\Omega}^{-1}\boldsymbol{X}\boldsymbol{A}^{-1}\boldsymbol{A}\boldsymbol{A}^{-1}\boldsymbol{X}^{\top}\boldsymbol{\Omega}^{-1}\\
&=\boldsymbol{\Omega}^{-1}-\boldsymbol{\Omega}^{-1}\boldsymbol{X}\boldsymbol{A}^{-1}\boldsymbol{X}^{\top}\boldsymbol{\Omega}^{-1}
 \end{align*}
and
\begin{align*}
\boldsymbol{N}\boldsymbol{N}= \boldsymbol{N} .
\end{align*}
Noting that
\begin{align*}
\boldsymbol{X}^\top \boldsymbol{N} \boldsymbol{X}= \boldsymbol{0}, \ \boldsymbol{X}^\top \boldsymbol{N} \boldsymbol{Z}=\boldsymbol{0}, \ \boldsymbol{Z}^\top \boldsymbol{N} \boldsymbol{Z}= \boldsymbol{Z}^\top \boldsymbol{Z},
\end{align*}
we see that the equalities \eqref{XX} and \eqref{XZ} hold; indeed, they can be rewritten as
\begin{align*}
&\boldsymbol{X}^{\top}(\boldsymbol{\Omega}^{-1}-\boldsymbol{\Omega}^{-1}\boldsymbol{X}\boldsymbol{A}^{-1}\boldsymbol{X}^{\top}\boldsymbol{\Omega}^{-1})\boldsymbol{X}=\boldsymbol{0}\\
\Leftrightarrow \quad &\boldsymbol{X}^{\top}\boldsymbol{\Omega}^{-1}\boldsymbol{X}-\boldsymbol{X}^\top \boldsymbol{\Omega}^{-1}\boldsymbol{X}\boldsymbol{A}^{-1}\boldsymbol{X}^{\top}\boldsymbol{\Omega}^{-1}\boldsymbol{X}=\boldsymbol{0}\\
\Leftrightarrow \quad & \boldsymbol{A}-\boldsymbol{A} \boldsymbol{A}^{-1} \boldsymbol{A} =\boldsymbol{0}
\end{align*}
and as
\begin{align*}
&\boldsymbol{X}^{\top}(\boldsymbol{\Omega}^{-1}-\boldsymbol{\Omega}^{-1}\boldsymbol{X}\boldsymbol{A}^{-1}\boldsymbol{X}^{\top}\boldsymbol{\Omega}^{-1})\boldsymbol{Z}=\boldsymbol{0}\\
\Leftrightarrow \quad &\boldsymbol{X}^{\top}\boldsymbol{\Omega}^{-1}\boldsymbol{Z}-\boldsymbol{X}^\top \boldsymbol{\Omega}^{-1}\boldsymbol{X}\boldsymbol{A}^{-1}\boldsymbol{X}^{\top}\boldsymbol{\Omega}^{-1}\boldsymbol{Z}=\boldsymbol{0}\\
\Leftrightarrow \quad &\boldsymbol{B}-\boldsymbol{A}\boldsymbol{A}^{-1}\boldsymbol{B}=\boldsymbol{0}.
\end{align*}
Finally, \eqref{ZZ} can be rewritten as
\begin{align*}
&\boldsymbol{Z}^{\top}(\boldsymbol{\Omega}^{-1}-\boldsymbol{\Omega}^{-1}\boldsymbol{X}\boldsymbol{A}^{-1}\boldsymbol{X}^{\top}\boldsymbol{\Omega}^{-1})\boldsymbol{Z}=\boldsymbol{Z}^\top \boldsymbol{Z} \\ 
\Leftrightarrow \quad &\boldsymbol{Z}^{\top}\boldsymbol{\Omega}^{-1}\boldsymbol{Z}-\boldsymbol{Z}^\top \boldsymbol{\Omega}^{-1}\boldsymbol{X}\boldsymbol{A}^{-1}\boldsymbol{X}^{\top}\boldsymbol{\Omega}^{-1}\boldsymbol{Z}=\boldsymbol{Z}^\top \boldsymbol{Z}\\
\Leftrightarrow \quad &\boldsymbol{D}-\boldsymbol{C}\boldsymbol{A}^{-1}\boldsymbol{B}=\boldsymbol{Z}^\top \boldsymbol{Z}\\ 
\Leftrightarrow \quad &\boldsymbol{\Delta}^{-1}+\boldsymbol{B}^\top \boldsymbol{A}^{-1}\boldsymbol{B}-\boldsymbol{B}^\top \boldsymbol{A}^{-1}\boldsymbol{B}=\boldsymbol{Z}^\top \boldsymbol{Z}\\
\Leftrightarrow \quad &\boldsymbol{\Delta}=(\boldsymbol{Z}^\top \boldsymbol{Z})^{-1}.
\end{align*}
This completes the proof.
\end{proof}

\begin{rem}\label{remkcond}
It can be directly shown  that \eqref{Kcond} is equivalent to $\boldsymbol{\Delta} = (\boldsymbol{Z}^\top \boldsymbol{Z})^{-1}$ in \eqref{Orep}.
See the following proof.
\end{rem}

\begin{proof}[Proof of Remark~\ref{remkcond}]
To see the sufficiency, suppose that  $\boldsymbol{\Delta}=(\boldsymbol{Z}^{\top}\boldsymbol{Z})^{-1}$ holds. From \eqref{Orep}, $\boldsymbol{\Omega}$ can be rewritten as
\begin{align*}
\boldsymbol{\Omega} &= \boldsymbol{X}\boldsymbol{\Gamma} \boldsymbol{X}^{\top}+\boldsymbol{Z}(\boldsymbol{Z}^{\top}\boldsymbol{Z})^{-1} \boldsymbol{Z}^{\top}+\boldsymbol{X}\boldsymbol{\Xi} \boldsymbol{Z}^{\top}+\boldsymbol{Z}\boldsymbol{\Xi} ^{\top}\boldsymbol{X}^{\top}\\
&= \boldsymbol{X} (\boldsymbol{X}^\top \boldsymbol{X})^{-1} \boldsymbol{X}^\top \boldsymbol{\Omega} \boldsymbol{X} (\boldsymbol{X}^\top \boldsymbol{X})^{-1} \boldsymbol{X}^{\top}+\boldsymbol{N} \\
&+ \boldsymbol{X}(\boldsymbol{X}^\top \boldsymbol{X})^{-1} \boldsymbol{X}^\top \boldsymbol{\Omega} \boldsymbol{Z} (\boldsymbol{Z}^\top \boldsymbol{Z})^{-1}\boldsymbol{Z}^\top+ \boldsymbol{Z}  (\boldsymbol{Z}^\top \boldsymbol{Z})^{-1} \boldsymbol{Z}^\top \boldsymbol{\Omega} \boldsymbol{X} (\boldsymbol{X}^\top \boldsymbol{X})^{-1}\boldsymbol{X}^\top\\
&= (\boldsymbol{I}_n-\boldsymbol{N})\boldsymbol{\Omega} (\boldsymbol{I}_n-\boldsymbol{N}) +\boldsymbol{N} + (\boldsymbol{I}_n-\boldsymbol{N})\boldsymbol{\Omega} \boldsymbol{N} + \boldsymbol{N} \boldsymbol{\Omega} (\boldsymbol{I}_n-\boldsymbol{N})\\
&= \boldsymbol{\Omega} - \boldsymbol{N} \boldsymbol{\Omega} -\boldsymbol{\Omega} \boldsymbol{N}+\boldsymbol{N}\boldsymbol{\Omega} \boldsymbol{N} + \boldsymbol{N} +\boldsymbol{\Omega} \boldsymbol{N} - \boldsymbol{N} \boldsymbol{\Omega} \boldsymbol{N} + \boldsymbol{N} \boldsymbol{\Omega} - \boldsymbol{N} \boldsymbol{\Omega} \boldsymbol{N}\\
&= \boldsymbol{N} + \boldsymbol{\Omega} - \boldsymbol{N}\boldsymbol{\Omega} \boldsymbol{N}.
\end{align*}
By letting $\boldsymbol{\Lambda} = \boldsymbol{\Omega} $, we have \eqref{Kcond}.
Next, to see the necessity, suppose that $\boldsymbol{\Omega}$ is of the form \eqref{Kcond}. From \eqref{Orep}, the equality
\begin{align*}
\boldsymbol{X}\boldsymbol{\Gamma} \boldsymbol{X}^\top + \boldsymbol{X} \boldsymbol{\Xi} \boldsymbol{Z}^\top + \boldsymbol{Z} \boldsymbol{\Xi} ^\top \boldsymbol{X}^\top + \boldsymbol{Z} \boldsymbol{\Delta} \boldsymbol{Z}^\top = \boldsymbol{N} + \boldsymbol{\Lambda} - \boldsymbol{N} \boldsymbol{\Lambda} \boldsymbol{N}    
\end{align*}
holds. Then, premultiplying by $(\boldsymbol{Z}^\top \boldsymbol{Z})^{-1}\boldsymbol{Z}^\top$ and postmultiplying by $\boldsymbol{Z}(\boldsymbol{Z}^\top \boldsymbol{Z})^{-1}$ yield $\boldsymbol{\Delta}= (\boldsymbol{Z}^\top \boldsymbol{Z})^{-1}$.
This completes the proof.
\end{proof}

\section{Equality between two general ridge estimators}\label{sec:3}
In this section, using the results of \cite{RefTK20}, we give a necessary and sufficient condition for an alternative form of \eqref{R2E} and discuss the results. 
The following theorem, the first main result of this paper, provides a condition corresponding to \eqref{Kcond2}.

\begin{thm}\label{thm1}
For given $\boldsymbol{K}_1, \boldsymbol{K}_2 \in \mathcal{S}^N(k)$, the equality \eqref{R2E} holds if and only if
\begin{equation}\label{Mcond1}
 \boldsymbol{X} = \boldsymbol{\Omega} \boldsymbol{X} \boldsymbol{G}, \ \boldsymbol{K}_1 = \boldsymbol{K}_2 \boldsymbol{G}
\quad 
\mbox{for some} \ \ \boldsymbol{G} \in \mathcal{M}(k,k).
\end{equation}
\end{thm}

\begin{proof}
Suppose first that the equality \eqref{R2E} holds, which is equivalent to the covariance matrix $\boldsymbol{\Omega}$ being of the form \eqref{Rcond} for some $\boldsymbol{\Gamma} \in \mathcal{S}^+(k), \boldsymbol{\Delta} \in \mathcal{S}^+(n-k)$ satisfying
\begin{equation}\label{XXGK}
\boldsymbol{X}^\top \boldsymbol{X} \boldsymbol{\Gamma} \boldsymbol{K}_1=\boldsymbol{K}_2,   
\end{equation}
by using Theorem 1 of \cite{RefTK20}. Then, \eqref{XXGK} can be rewritten as
\begin{align*}
&\boldsymbol{X}^{\top}\boldsymbol{X}\boldsymbol{\Gamma} \boldsymbol{K}_1=\boldsymbol{K}_2\\
\Leftrightarrow \quad &\boldsymbol{K}_1 \boldsymbol{\Gamma} \boldsymbol{X}^{\top}\boldsymbol{X}=\boldsymbol{K}_2\\
\Leftrightarrow \quad &\boldsymbol{K}_1=\boldsymbol{K}_2(\boldsymbol{\Gamma} \boldsymbol{X}^{\top}\boldsymbol{X} )^{-1}.
\end{align*}
Let us denote $\boldsymbol{G}=(\boldsymbol{\Gamma} \boldsymbol{X}^{\top}\boldsymbol{X} )^{-1}$. 
Since the inverse of the covariance matrix $\boldsymbol{\Omega}$ can be expressed as
\begin{align*}
\boldsymbol{\Omega}^{-1}=\boldsymbol{X}(\boldsymbol{X}^{\top}\boldsymbol{X})^{-1}\boldsymbol{\Gamma}^{-1}(\boldsymbol{X}^{\top}\boldsymbol{X})^{-1}\boldsymbol{X}^{\top}+\boldsymbol{Z}(\boldsymbol{Z}^{\top}\boldsymbol{Z})^{-1}\boldsymbol{\Delta}^{-1}(\boldsymbol{Z}^{\top}\boldsymbol{Z})^{-1}\boldsymbol{Z}^{\top},
\end{align*}
the following equality holds:
\begin{align*}
    \boldsymbol{\Omega}^{-1}\boldsymbol{X}=\boldsymbol{X}(\boldsymbol{X}^{\top}\boldsymbol{X})^{-1}\boldsymbol{\Gamma}^{-1}=\boldsymbol{X}(\boldsymbol{\Gamma} \boldsymbol{X}^{\top}\boldsymbol{X})^{-1}=\boldsymbol{X}\boldsymbol{G}.
\end{align*}
Therefore, we have \eqref{Mcond1}.

Conversely, suppose that \eqref{Mcond1} holds.
Then, we have
\begin{align*}
\boldsymbol{X}^{\top}\boldsymbol{\Omega}^{-1}&=\boldsymbol{G}^{\top}\boldsymbol{X}^{\top}\\
&= \boldsymbol{G}^{\top}(\boldsymbol{X}^{\top}\boldsymbol{X}+\boldsymbol{K}_2)(\boldsymbol{X}^{\top}\boldsymbol{X}+\boldsymbol{K}_2)^{-1}\boldsymbol{X}^{\top}\\
&= (\boldsymbol{G}^{\top}\boldsymbol{X}^{\top}\boldsymbol{X}+\boldsymbol{G}^{\top}\boldsymbol{K}_2)(\boldsymbol{X}^{\top}\boldsymbol{X}+\boldsymbol{K}_2)^{-1}\boldsymbol{X}^{\top}\\
&= (\boldsymbol{X}^{\top}\boldsymbol{\Omega}^{-1}\boldsymbol{X}+\boldsymbol{K}_1)(\boldsymbol{X}^{\top}\boldsymbol{X}+\boldsymbol{K}_2)^{-1}\boldsymbol{X}^{\top}.
\end{align*}
Thus, this equality is equivalent to
\begin{align*}
(\boldsymbol{X}^{\top}\boldsymbol{\Omega}^{-1}\boldsymbol{X}+\boldsymbol{K}_1)^{-1}\boldsymbol{X}^\top \boldsymbol{\Omega}^{-1}=(\boldsymbol{X}^{\top}\boldsymbol{X}+\boldsymbol{K}_2)^{-1}\boldsymbol{X}^{\top}.
\end{align*}
This completes the proof.
\end{proof}

\begin{rem}
Let $\boldsymbol{K}_1 = \boldsymbol{K}_2 = \boldsymbol{0}$. 
Then \eqref{Mcond1} becomes the condition \eqref{Kcond2}, which is equivalent to \eqref{Kcond1}.
\end{rem}

\begin{rem}
From Theorem~\ref{thm1}, if \eqref{R2E} holds, then we have
\[
\MC(\boldsymbol{X}) = \MC(\boldsymbol{\Omega} \boldsymbol{X}) , \ \MC(\boldsymbol{K}_1) = \MC(\boldsymbol{K}_2) .
\]
However, the converse is generally not true.
\end{rem}

\begin{rem}
If there exists a matrix $\boldsymbol{G}$ satisfying \eqref{Mcond1}, then $\boldsymbol{G}$ is nonsingular.
\end{rem}

If $\boldsymbol{K}_1$ and $\boldsymbol{K}_2$ are positive definite, we directly obtain the following corollary by using Theorem~\ref{thm1}.

\begin{cor}\label{cormtpd}
(i) For given $\boldsymbol{K}_1, \boldsymbol{K}_2 \in \mathcal{S}^+(k)$, the equality \eqref{R2E} holds if and only if $\boldsymbol{X} = \boldsymbol{\Omega} \boldsymbol{X} \boldsymbol{K}_2^{-1} \boldsymbol{K}_1$.\\
(ii) For given $\boldsymbol{K} \in \mathcal{S}^+(k)$, 
\[
 \hat{\boldsymbol{\beta}}(\boldsymbol{\Omega}, \boldsymbol{K}) = \hat{\boldsymbol{\beta}} (\boldsymbol{I}_n, \boldsymbol{K}) \quad \mbox{for any} \ \ \boldsymbol{y} \in \mathbb{R}^n 
\]
holds if and only if $\boldsymbol{X} = \boldsymbol{\Omega} \boldsymbol{X}$.
\end{cor}
\begin{rem}
Let $\boldsymbol{K}=\lambda \boldsymbol{I}_k \in \mathcal{S}^+(k)$. 
Corollary~\ref{cormtpd} (ii) implies that 
\[ \hat{\boldsymbol{\beta}}(\boldsymbol{\Omega},\lambda \boldsymbol{I}_k)=\hat{\boldsymbol{\beta}}(\boldsymbol{I}_n, \lambda \boldsymbol{I}_k)\quad \mbox{for any} \ \ \boldsymbol{y} \in \mathbb{R}^n \]
holds if and only if $\boldsymbol{X}=\boldsymbol{\Omega} \boldsymbol{X}$. 
See also Section~\ref{sec:5}.
\end{rem}

Furthermore, if $\boldsymbol{K}_1$ and $\boldsymbol{K}_2$ are idempotent matrices, we derive the following proposition.

\begin{prop}
Let $\boldsymbol{K}_1 \in \mathcal{S}^N(k)$ and $\boldsymbol{K}_2 \in \mathcal{S}^N(k)$ be idempotent matrices. If the equality \eqref{R2E} holds, then the equality \eqref{OBE} holds and $\boldsymbol{K}_1 = \boldsymbol{K}_2$.
\end{prop}
\begin{proof}
Suppose that the equality \eqref{R2E} holds, which is equivalent to the covariance matrix $\boldsymbol{\Omega}$ being of the form \eqref{Rcond} for some $\boldsymbol{\Gamma} \in \mathcal{S}^+(k), \boldsymbol{\Delta} \in \mathcal{S}^+(n-k)$ satisfying \eqref{XXGK} by using Theorem 1 of \cite{RefTK20}. 
First,  \eqref{Rcond} is equivalent to \eqref{OBE} from \cite{RefR67}. 
Since \eqref{XXGK} is equivalent to
\begin{align*}
\boldsymbol{K}_1 = \boldsymbol{K}_2(\boldsymbol{\Gamma} \boldsymbol{X}^\top \boldsymbol{X})^{-1},
\end{align*}
it holds that
\begin{align*}
& \boldsymbol{X}^\top \boldsymbol{X} \boldsymbol{\Gamma} \boldsymbol{K}_1 = \boldsymbol{K}_2  \\
\Leftrightarrow \quad & \boldsymbol{X}^\top \boldsymbol{X} \boldsymbol{\Gamma} \boldsymbol{K}_1^2 = \boldsymbol{K}_2  \\
\Rightarrow \quad & \boldsymbol{K}_2 \boldsymbol{K}_1 = \boldsymbol{K}_2 \\
\Rightarrow \quad & \boldsymbol{K}_2^2 (\boldsymbol{\Gamma} \boldsymbol{X}^\top \boldsymbol{X})^{-1} = \boldsymbol{K}_2\\
\Leftrightarrow \quad &  \boldsymbol{K}_2 (\boldsymbol{\Gamma} \boldsymbol{X}^\top \boldsymbol{X})^{-1} = \boldsymbol{K}_2\\
\Rightarrow \quad &  \boldsymbol{K}_1 = \boldsymbol{K}_2.
\end{align*}
 This completes the proof.
\end{proof}

As mentioned in Section~\ref{sec:1}, the bias of the general ridge estimator $\hat{\boldsymbol{\beta}}(\boldsymbol{\Phi},\boldsymbol{K})$
\[ 
\E\left[ \hat{\boldsymbol{\beta}}(\boldsymbol{\Phi},\boldsymbol{K}) \right] - \boldsymbol{\beta} 
= (\boldsymbol{X}^\top \boldsymbol{\Phi}^{-1} \boldsymbol{X} + \boldsymbol{K})^{-1} \boldsymbol{X}^\top \boldsymbol{\Phi}^{-1} \boldsymbol{X} \boldsymbol{\beta} - \boldsymbol{\beta}
= -(\boldsymbol{X}^\top \boldsymbol{\Phi}^{-1} \boldsymbol{X} + \boldsymbol{K})^{-1} \boldsymbol{K} \boldsymbol{\beta} 
\]
is generally nonzero. 
An implication of the condition for \eqref{R2E} to hold becomes more apparent by considering conditions under which the biases of the two general ridge estimators $\hat{\boldsymbol{\beta}}(\boldsymbol{\Omega},\boldsymbol{K}_1)$ and $\hat{\boldsymbol{\beta}}(\boldsymbol{I}_n,\boldsymbol{K}_2)$ coincide, that is to say, 
\begin{equation}\label{biascond}
\E\left[ \hat{\boldsymbol{\beta}}(\boldsymbol{\Omega},\boldsymbol{K}_1) \right]
= \E\left[\hat{\boldsymbol{\beta}}(\boldsymbol{I}_n,\boldsymbol{K}_2)\right]
\quad \mbox{for any} \ \ \boldsymbol{\beta} \in \mathbb{R}^k.
\end{equation}
If \eqref{R2E} holds, then \eqref{biascond} holds. 
But the converse is not true in general. 
First, we derive the necessary and sufficient condition under which \eqref{biascond} holds.

\begin{thm}\label{thmbias}
Let $\boldsymbol{\Omega}$ be expressed as \eqref{Orep}.
For given $\boldsymbol{K}_1, \boldsymbol{K}_2 \in \mathcal{S}^N(k)$, the equality \eqref{biascond} holds if and only if
\begin{align}\label{Th33}
\boldsymbol{X}^\top \boldsymbol{X} \boldsymbol{A} ^{-1}\boldsymbol{K}_1=\boldsymbol{K}_2,
\end{align}
where $\boldsymbol{A}=(\boldsymbol{\Gamma}- \boldsymbol{\Xi} \boldsymbol{\Delta}^{-1}\boldsymbol{\Xi}^\top)^{-1}$.
\end{thm}

\begin{proof}
First, the expectation values of $\hat{\boldsymbol{\beta}}(\boldsymbol{\Omega},\boldsymbol{K}_1)$ and $\hat{\boldsymbol{\beta}}(\boldsymbol{I}_n,\boldsymbol{K}_2)$ are calculated to be
\begin{align*}
& \E\left[\hat{\boldsymbol{\beta}}(\boldsymbol{\Omega},\boldsymbol{K}_1)\right] = (\boldsymbol{A} + \boldsymbol{K}_1)^{-1} \boldsymbol{A} \boldsymbol{\beta} = (\boldsymbol{I}_k + \boldsymbol{A}^{-1} \boldsymbol{K}_1)^{-1} \boldsymbol{\beta},\\
& \E\left[\hat{\boldsymbol{\beta}}(\boldsymbol{I}_n,\boldsymbol{K}_2)\right] = (\boldsymbol{X}^\top \boldsymbol{X} + \boldsymbol{K}_2)^{-1} \boldsymbol{X}^\top \boldsymbol{X} \boldsymbol{\beta} = \left\{\boldsymbol{I}_k + ( \boldsymbol{X}^\top \boldsymbol{X} )^{-1}\boldsymbol{K}_2\right\}^{-1}  \boldsymbol{\beta}.
\end{align*}
So, \eqref{biascond} is equivalent to
\begin{align*}
&(\boldsymbol{I}_k + \boldsymbol{A}^{-1} \boldsymbol{K}_1)^{-1} =  \left\{\boldsymbol{I}_k + ( \boldsymbol{X}^\top \boldsymbol{X} )^{-1}\boldsymbol{K}_2\right\}^{-1}\\
\Leftrightarrow \quad & \boldsymbol{I}_k + \boldsymbol{A}^{-1} \boldsymbol{K}_1 = \boldsymbol{I}_k + ( \boldsymbol{X}^\top \boldsymbol{X} )^{-1}\boldsymbol{K}_2\\
\Leftrightarrow \quad &  \boldsymbol{X}^\top \boldsymbol{X} \boldsymbol{A}^{-1} \boldsymbol{K}_1 = \boldsymbol{K}_2.
\end{align*}
This completes the proof.
\end{proof}
By using Theorem~\ref{thmbias}, we see that the necessary and sufficient condition for \eqref{R2E} to hold is given by \eqref{Rcond} and the equivalence of biases, which gives a clearer interpretation compared with the result of \cite{RefTK20}. 
The following theorem is the second main result of this paper.

\begin{thm}\label{thm3}
For given $\boldsymbol{K}_1, \boldsymbol{K}_2 \in \mathcal{S}^N(k)$, the equality \eqref{R2E} holds if and only if $\boldsymbol{\Omega}$ is of the form \eqref{Rcond} and the equality \eqref{biascond} holds.
\end{thm}

Moreover, in the following proposition, we consider necessary and sufficient conditions under which \eqref{biascond} and the equality
\begin{align}\label{Covcond}
\V\left(\hat{\boldsymbol{\beta}}(\boldsymbol{\Omega},\boldsymbol{K}_1)\right)=\V\left(\hat{\boldsymbol{\beta}}(\boldsymbol{I}_n,\boldsymbol{K}_2)\right)
\end{align}
both hold. 

\begin{prop} \label{pr36}
Let $\boldsymbol{\Omega}$ be expressed in \eqref{Orep}. 
For given $\boldsymbol{K}_1,\boldsymbol{K}_2 \in \mathcal{S}^N(k)$,
\eqref{biascond} and \eqref{Covcond} hold
if and only if the following two conditions both hold:
\begin{align}
&\boldsymbol{X}^\top \boldsymbol{X} \boldsymbol{A}^{-1}\boldsymbol{K}_1=\boldsymbol{K}_2, \notag\\
& \boldsymbol{A} + \boldsymbol{K}_1 \boldsymbol{\Xi} \boldsymbol{\Delta}^{-1} \boldsymbol{\Xi} ^\top \boldsymbol{K}_1=(\boldsymbol{I}_k-\boldsymbol{K}_1 \boldsymbol{\Xi} \boldsymbol{\Delta}^{-1} \boldsymbol{\Xi} ^\top)\boldsymbol{\Gamma}^{-1}(\boldsymbol{I}_k- \boldsymbol{\Xi} \boldsymbol{\Delta}^{-1} \boldsymbol{\Xi} ^\top \boldsymbol{K}_1), \label{Pr35}
\end{align}
where $\boldsymbol{A}=(\boldsymbol{\Gamma}- \boldsymbol{\Xi} \boldsymbol{\Delta}^{-1}\boldsymbol{\Xi}^\top)^{-1}$.
\end{prop}
\begin{proof} 
From Theorem~\ref{thmbias}, \eqref{biascond} is equivalent to \eqref{Th33}. 
Next, the covariance matrices of $\hat{\boldsymbol{\beta}}(\boldsymbol{\Omega},\boldsymbol{K}_1)$ and $\hat{\boldsymbol{\beta}}(\boldsymbol{I}_n,\boldsymbol{K}_2)$ are calculated to be
\begin{align*}
\V\left(\hat{\boldsymbol{\beta}}(\boldsymbol{\Omega},\boldsymbol{K}_1)\right)
&=\V\left((\boldsymbol{A}+\boldsymbol{K}_1)^{-1}\boldsymbol{X}^\top \boldsymbol{\Omega}^{-1}\boldsymbol{y}\right)\\
&= \sigma^2 (\boldsymbol{A}+\boldsymbol{K}_1)^{-1}\boldsymbol{X}^\top \boldsymbol{\Omega}^{-1}\boldsymbol{\Omega} \boldsymbol{\Omega}^{-1} \boldsymbol{X}(\boldsymbol{A}+\boldsymbol{K}_1)^{-1}\\
&=\sigma^2(\boldsymbol{A}+\boldsymbol{K}_1)^{-1}\boldsymbol{A}(\boldsymbol{A}+\boldsymbol{K}_1)^{-1},\\
\V\left(\hat{\boldsymbol{\beta}}(\boldsymbol{I}_n,\boldsymbol{K}_2)\right)
&=\V\left((\boldsymbol{X}^\top \boldsymbol{X}+\boldsymbol{K}_2)^{-1}\boldsymbol{X}^\top \boldsymbol{y}\right)\\
&= \sigma^2 (\boldsymbol{X}^\top \boldsymbol{X}+\boldsymbol{K}_2)^{-1}\boldsymbol{X}^\top \boldsymbol{\Omega} \boldsymbol{X} (\boldsymbol{X}^\top \boldsymbol{X}+\boldsymbol{K}_2)^{-1}\\
&= \sigma^2 (\boldsymbol{X}^\top \boldsymbol{X}+\boldsymbol{K}_2)^{-1}\boldsymbol{X}^\top \boldsymbol{X} \boldsymbol{\Gamma} \boldsymbol{X}^\top \boldsymbol{X} (\boldsymbol{X}^\top \boldsymbol{X}+\boldsymbol{K}_2)^{-1},
\end{align*}
respectively. 
Since
\begin{align*}
\boldsymbol{K}_2= \boldsymbol{X}^\top \boldsymbol{X} \boldsymbol{A}^{-1}\boldsymbol{K}_1 = \boldsymbol{K}_1 \boldsymbol{A}^{-1}\boldsymbol{X}^\top \boldsymbol{X} ,
\end{align*}
the covariance matrix of $\hat{\boldsymbol{\beta}}(\boldsymbol{I}_n,\boldsymbol{K}_2)$ can be further rewritten as
\begin{align*}
\V\left(\hat{\boldsymbol{\beta}}(\boldsymbol{I}_n,\boldsymbol{K}_2)\right)&= \sigma^2 (\boldsymbol{I}_k+\boldsymbol{A}^{-1}\boldsymbol{K}_1)^{-1}\boldsymbol{\Gamma} (\boldsymbol{I}_k+\boldsymbol{K}_1 \boldsymbol{A}^{-1})^{-1}.
\end{align*}
Thus, \eqref{Covcond} is equivalent to
\begin{align*}
 & (\boldsymbol{A}+\boldsymbol{K}_1)^{-1}\boldsymbol{A}(\boldsymbol{A}+\boldsymbol{K}_1)^{-1} = (\boldsymbol{I}_k+\boldsymbol{A}^{-1}\boldsymbol{K}_1)^{-1}\boldsymbol{\Gamma} (\boldsymbol{I}_k+\boldsymbol{K}_1 \boldsymbol{A}^{-1})^{-1}\\
\Leftrightarrow \quad & (\boldsymbol{A}+\boldsymbol{K}_1) \boldsymbol{A}^{-1} (\boldsymbol{A}+\boldsymbol{K}_1) = (\boldsymbol{I}_k+\boldsymbol{K}_1 \boldsymbol{A}^{-1}) \boldsymbol{\Gamma} ^{-1} (\boldsymbol{I}_k+\boldsymbol{A}^{-1}\boldsymbol{K}_1)\\
\Leftrightarrow \quad & \boldsymbol{A}+ 2\boldsymbol{K}_1 + \boldsymbol{K}_1 \boldsymbol{A}^{-1}\boldsymbol{K}_1= \boldsymbol{\Gamma}^{-1}+\boldsymbol{K}_1 \boldsymbol{A}^{-1}\boldsymbol{\Gamma}^{-1}  + \boldsymbol{\Gamma}^{-1} \boldsymbol{A}^{-1} \boldsymbol{K}_1 + \boldsymbol{K}_1 \boldsymbol{A}^{-1} \boldsymbol{\Gamma}^{-1} \boldsymbol{A}^{-1} \boldsymbol{K}_1 \\
\Leftrightarrow \quad & \boldsymbol{A} + 2\boldsymbol{K}_1 + \boldsymbol{K}_1(\boldsymbol{\Gamma} - \boldsymbol{\Xi} \boldsymbol{\Delta}^{-1} \boldsymbol{\Xi}^{\top})\boldsymbol{K}_1\\
&= \boldsymbol{\Gamma}^{-1}+\boldsymbol{K}_1 (\boldsymbol{\Gamma} - \boldsymbol{\Xi} \boldsymbol{\Delta}^{-1} \boldsymbol{\Xi}^{\top}) \boldsymbol{\Gamma}^{-1}+ \boldsymbol{\Gamma}^{-1}(\boldsymbol{\Gamma} - \boldsymbol{\Xi} \boldsymbol{\Delta}^{-1} \boldsymbol{\Xi}^{\top})\boldsymbol{K}_1\\
&+\boldsymbol{K}_1(\boldsymbol{\Gamma} - \boldsymbol{\Xi} \boldsymbol{\Delta}^{-1} \boldsymbol{\Xi}^{\top}) \boldsymbol{\Gamma}^{-1}(\boldsymbol{\Gamma} - \boldsymbol{\Xi} \boldsymbol{\Delta}^{-1} \boldsymbol{\Xi}^{\top})\boldsymbol{K}_1\\
\Leftrightarrow \quad & \boldsymbol{A}+ 2\boldsymbol{K}_1 +\boldsymbol{K}_1 \boldsymbol{\Gamma} \boldsymbol{K}_1- \boldsymbol{K}_1 \boldsymbol{\Xi} \boldsymbol{\Delta}^{-1} \boldsymbol{\Xi}^{\top} \boldsymbol{K}_1 \\
& = \boldsymbol{\Gamma}^{-1}+ \boldsymbol{K}_1 - \boldsymbol{K}_1  \boldsymbol{\Xi} \boldsymbol{\Delta}^{-1} \boldsymbol{\Xi}^{\top} \boldsymbol{\Gamma}^{-1}+\boldsymbol{K}_1- \boldsymbol{\Gamma}^{-1}  \boldsymbol{\Xi} \boldsymbol{\Delta}^{-1} \boldsymbol{\Xi}^{\top}  \boldsymbol{K}_1 \\
&+  \boldsymbol{K}_1 \boldsymbol{\Gamma} \boldsymbol{K}_1 -2 \boldsymbol{K}_1  \boldsymbol{\Xi} \boldsymbol{\Delta}^{-1} \boldsymbol{\Xi}^{\top} \boldsymbol{K}_1 
 + \boldsymbol{K}_1  \boldsymbol{\Xi} \boldsymbol{\Delta}^{-1} \boldsymbol{\Xi}^{\top}  \boldsymbol{\Gamma}^{-1}  \boldsymbol{\Xi} \boldsymbol{\Delta}^{-1} \boldsymbol{\Xi}^{\top}  \boldsymbol{K}_1\\
\Leftrightarrow \quad & \boldsymbol{A}+\boldsymbol{K}_1  \boldsymbol{\Xi} \boldsymbol{\Delta}^{-1} \boldsymbol{\Xi}^{\top} \boldsymbol{K}_1 \\
&=\boldsymbol{\Gamma}^{-1}-\boldsymbol{K}_1  \boldsymbol{\Xi} \boldsymbol{\Delta}^{-1} \boldsymbol{\Xi}^{\top} \boldsymbol{\Gamma}^{-1}- \boldsymbol{\Gamma}^{-1}  \boldsymbol{\Xi} \boldsymbol{\Delta}^{-1} \boldsymbol{\Xi}^{\top}  \boldsymbol{K}_1 + \boldsymbol{K}_1  \boldsymbol{\Xi} \boldsymbol{\Delta}^{-1} \boldsymbol{\Xi}^{\top}  \boldsymbol{\Gamma}^{-1}  \boldsymbol{\Xi} \boldsymbol{\Delta}^{-1} \boldsymbol{\Xi}^{\top}  \boldsymbol{K}_1\\
\Leftrightarrow \quad & \boldsymbol{A}+\boldsymbol{K}_1  \boldsymbol{\Xi} \boldsymbol{\Delta}^{-1} \boldsymbol{\Xi}^{\top} \boldsymbol{K}_1= (\boldsymbol{I}_k-\boldsymbol{K}_1 \boldsymbol{\Xi} \boldsymbol{\Delta}^{-1} \boldsymbol{\Xi} ^\top)\boldsymbol{\Gamma}^{-1}(\boldsymbol{I}_k- \boldsymbol{\Xi} \boldsymbol{\Delta}^{-1} \boldsymbol{\Xi} ^\top \boldsymbol{K}_1).
\end{align*}
This completes the proof.
\end{proof} 

\cite{RefTK20} proposed to use the rank of the $L^2$ difference matrix
\begin{align*}
d_1(\boldsymbol{\Omega},\boldsymbol{K}_1,\boldsymbol{K}_2)=\E\left[\left(\hat{\boldsymbol{\beta}}(\boldsymbol{I}_n,\boldsymbol{K}_2)-\hat{\boldsymbol{\beta}}(\boldsymbol{\Omega},\boldsymbol{K}_1)\right)\left(\hat{\boldsymbol{\beta}}(\boldsymbol{I}_n,\boldsymbol{K}_2)-\hat{\boldsymbol{\beta}}(\boldsymbol{\Omega},\boldsymbol{K}_1)\right)^\top\right] 
\end{align*}
as a measure of the difference between two general ridge estimators $\hat{\boldsymbol{\beta}}(\boldsymbol{\Omega},\boldsymbol{K}_1)$ and $\hat{\boldsymbol{\beta}}(\boldsymbol{I}_n,\boldsymbol{K}_2)$. 
The following corollary provides a condition under which 
\begin{align*}
\rk\left(d_1(\boldsymbol{\Omega},\boldsymbol{K}_1,\boldsymbol{K}_2)\right) = 0.   
\end{align*}

\begin{cor}
Let $\boldsymbol{A}=(\boldsymbol{\Gamma}- \boldsymbol{\Xi} \boldsymbol{\Delta}^{-1}\boldsymbol{\Xi}^\top)^{-1}$.
From Proposition~\ref{pr36}, the two conditions \eqref{Th33} and \eqref{Pr35} both hold if and only if $d_1(\boldsymbol{\Omega},\boldsymbol{K}_1,\boldsymbol{K}_2)=\boldsymbol{0}$.
\end{cor}

\section{Equality between residual sums of squares}\label{sec:4}
In this section, we derive a necessary and sufficient condition for \eqref{GRSE}, which is the third main result of this paper.

\begin{thm}\label{thm2}
Let $\boldsymbol{\Omega}$ be expressed in \eqref{Orep}.
For given $\boldsymbol{K}_1, \boldsymbol{K}_2 \in \mathcal{S}^N(k)$, the equality \eqref{GRSE} holds if and only if the following three conditions simultaneously hold:
\begin{align}
& \boldsymbol{K}_1 (\boldsymbol{A} + \boldsymbol{K}_1)^{-1} \boldsymbol{A} (\boldsymbol{A} + \boldsymbol{K}_1)^{-1} \boldsymbol{K}_1 \notag \\
&= \boldsymbol{K}_2(\boldsymbol{X}^\top \boldsymbol{X} + \boldsymbol{K}_2)^{-1} \boldsymbol{X}^\top \boldsymbol{X} (\boldsymbol{X}^\top \boldsymbol{X} + \boldsymbol{K}_2)^{-1} \boldsymbol{K}_2, \label{Mcond21} \\
& \boldsymbol{K}_2(\boldsymbol{X}^\top \boldsymbol{X} + \boldsymbol{K}_2)^{-1} \boldsymbol{K}_2 \boldsymbol{\Xi} = \boldsymbol{0}, \label{Mcond22} \\
& \boldsymbol{\Delta} = (\boldsymbol{Z}^\top \boldsymbol{Z})^{-1}, \label{Mcond23}
\end{align}
where $\boldsymbol{A} = (\boldsymbol{\Gamma} - \boldsymbol{\Xi} \boldsymbol{\Delta}^{-1} \boldsymbol{\Xi}^\top)^{-1}$.
\end{thm}

\begin{proof}
For simplicity, the notation in \eqref{ABCD} is used. 
Let us denote
\begin{align*}
\boldsymbol{S}=\boldsymbol{I}_n-\boldsymbol{X}(\boldsymbol{X}^{\top}\boldsymbol{\Omega}^{-1}\boldsymbol{X}+\boldsymbol{K}_1)^{-1}\boldsymbol{X}^{\top}\boldsymbol{\Omega}^{-1}, \
\boldsymbol{T}=\boldsymbol{I}_n-\boldsymbol{X}(\boldsymbol{X}^{\top}\boldsymbol{X}+\boldsymbol{K}_2)^{-1}\boldsymbol{X}^{\top}.    
\end{align*}
Since
\begin{align*}
\boldsymbol{y}-\boldsymbol{X}\hat{\boldsymbol{\beta}}(\boldsymbol{\Omega},\boldsymbol{K}_1)
&= \boldsymbol{y}-\boldsymbol{X}(\boldsymbol{X}^{\top}\boldsymbol{\Omega}^{-1}\boldsymbol{X}+\boldsymbol{K}_1)^{-1}\boldsymbol{X}^{\top}\boldsymbol{\Omega}^{-1}\boldsymbol{y}\\
&= \left\{\boldsymbol{I}_n-\boldsymbol{X}(\boldsymbol{X}^{\top}\boldsymbol{\Omega}^{-1}\boldsymbol{X}+\boldsymbol{K}_1)^{-1}\boldsymbol{X}^{\top}\boldsymbol{\Omega}^{-1}\right\}\boldsymbol{y}
=\boldsymbol{S}\boldsymbol{y}
\end{align*}
and
\begin{align*}
\boldsymbol{y}-\boldsymbol{X}\hat{\boldsymbol{\beta}}(\boldsymbol{I}_n,\boldsymbol{K}_2)&= \boldsymbol{y}-\boldsymbol{X}(\boldsymbol{X}^{\top}\boldsymbol{X}+\boldsymbol{K}_2)^{-1}\boldsymbol{X}^{\top}\boldsymbol{y}\\
&=\left\{\boldsymbol{I}_n-\boldsymbol{X}(\boldsymbol{X}^{\top}\boldsymbol{X}+\boldsymbol{K}_2)^{-1}\boldsymbol{X}^{\top}\right\}\boldsymbol{y}
=\boldsymbol{T}\boldsymbol{y},
\end{align*}
the residual sums of squares $RSS(\boldsymbol{\Omega},\boldsymbol{K}_1)$ and $RSS(\boldsymbol{I}_n,\boldsymbol{K}_2)$ can be rewritten as
\begin{align*}
&RSS(\boldsymbol{\Omega},\boldsymbol{K}_1)=\left(\boldsymbol{y}-\boldsymbol{X}\hat{{\boldsymbol{\beta}}}(\boldsymbol{\Omega},\boldsymbol{K}_1)\right)^{\top}\boldsymbol{\Omega}^{-1}\left(\boldsymbol{y}-\boldsymbol{X}\hat{{\boldsymbol{\beta}}}(\boldsymbol{\Omega},\boldsymbol{K}_1)\right)=\boldsymbol{y}^{\top}\boldsymbol{S}^{\top}\boldsymbol{\Omega}^{-1} \boldsymbol{S}\boldsymbol{y}
\end{align*}
and
\begin{align*}
&RSS(\boldsymbol{I}_n,\boldsymbol{K}_2)=\left(\boldsymbol{y}-\boldsymbol{X}\hat{{\boldsymbol{\beta}}}(\boldsymbol{I}_n,\boldsymbol{K}_2)\right)^{\top} \left(\boldsymbol{y}-\boldsymbol{X}\hat{{\boldsymbol{\beta}}}(\boldsymbol{I}_n,\boldsymbol{K}_2)\right)=\boldsymbol{y}^{\top}\boldsymbol{T}\boldsymbol{T} \boldsymbol{y} ,
\end{align*}
respectively.
Thus, the problem is to derive a condition under which $\boldsymbol{y}^\top (\boldsymbol{S}^{\top}\boldsymbol{\Omega}^{-1} \boldsymbol{S}-\boldsymbol{T}\boldsymbol{T}  )\boldsymbol{y}=0$ holds for any $\boldsymbol{y}\in \mathbb{R}^{n}$, which is equivalent to deriving a condition for
\begin{equation}\label{RSSK}
 \boldsymbol{S}^{\top}\boldsymbol{\Omega}^{-1} \boldsymbol{S}-\boldsymbol{T}\boldsymbol{T}=\boldsymbol{0} .
\end{equation}
The equality \eqref{RSSK}  holds if and only if the following three equalities simultaneously hold:
\begin{align}
\boldsymbol{X}^\top (\boldsymbol{S}^{\top}\boldsymbol{\Omega}^{-1} \boldsymbol{S}-\boldsymbol{T}\boldsymbol{T}  ) \boldsymbol{X} =\boldsymbol{0}, \label{XXK} \\
\boldsymbol{X}^\top (\boldsymbol{S}^{\top}\boldsymbol{\Omega}^{-1} \boldsymbol{S}-\boldsymbol{T}\boldsymbol{T}  )  \boldsymbol{Z} =\boldsymbol{0}, \label{XZK} \\
\boldsymbol{Z}^\top (\boldsymbol{S}^{\top}\boldsymbol{\Omega}^{-1} \boldsymbol{S}-\boldsymbol{T}\boldsymbol{T}  ) \boldsymbol{Z} =\boldsymbol{0}. \label{ZZK}
\end{align}
The quantities $\boldsymbol{X}^{\top}\boldsymbol{S}^{\top}\boldsymbol{\Omega}^{-1}\boldsymbol{S}\boldsymbol{X},\boldsymbol{X}^{\top}\boldsymbol{S}^{\top}\boldsymbol{\Omega}^{-1}\boldsymbol{S}\boldsymbol{Z}$, and $\boldsymbol{Z}^{\top}\boldsymbol{S}^{\top}\boldsymbol{\Omega}^{-1}\boldsymbol{S}\boldsymbol{Z}$ are calculated as follows:
\begin{align*}
&\boldsymbol{X}^{\top}\boldsymbol{S}^{\top}\boldsymbol{\Omega}^{-1}\boldsymbol{S}\boldsymbol{X}\\
& =\boldsymbol{X}^{\top}\{\boldsymbol{I}_n-\boldsymbol{\Omega}^{-1}\boldsymbol{X}(\boldsymbol{X}^{\top}\boldsymbol{\Omega}^{-1}\boldsymbol{X}+\boldsymbol{K}_1)^{-1}\boldsymbol{X}^{\top}\}\boldsymbol{\Omega}^{-1}\{\boldsymbol{I}_n-\boldsymbol{X}(\boldsymbol{X}^{\top}\boldsymbol{\Omega}^{-1}\boldsymbol{X}+\boldsymbol{K}_1)^{-1}\boldsymbol{X}^{\top}\boldsymbol{\Omega}^{-1}\}\boldsymbol{X} \\ 
&=\{\boldsymbol{X}^{\top}-\boldsymbol{A}(\boldsymbol{A}+\boldsymbol{K}_1)^{-1}\boldsymbol{X}^{\top}\}\boldsymbol{\Omega}^{-1}\{\boldsymbol{X}-\boldsymbol{X}(\boldsymbol{A}+\boldsymbol{K}_1)^{-1}\boldsymbol{A}\} \\
&=\{\boldsymbol{I}_k-\boldsymbol{A}(\boldsymbol{A}+\boldsymbol{K}_1)^{-1}\}\boldsymbol{X}^{\top}\boldsymbol{\Omega}^{-1} \boldsymbol{X}\{\boldsymbol{I}_k-(\boldsymbol{A}+\boldsymbol{K}_1)^{-1}\boldsymbol{A}\} \\
&=\{\boldsymbol{I}_k-\boldsymbol{A}(\boldsymbol{A}+\boldsymbol{K}_1)^{-1}\}\boldsymbol{A}\{\boldsymbol{I}_k-(\boldsymbol{A}+\boldsymbol{K}_1)^{-1}\boldsymbol{A}\},
\end{align*}
\begin{align}\label{EQA}
&\boldsymbol{X}^{\top}\boldsymbol{S}^{\top}\boldsymbol{\Omega}^{-1}\boldsymbol{S}\boldsymbol{Z} \notag\\
&=\boldsymbol{X}^{\top}\{\boldsymbol{I}_n-\boldsymbol{\Omega}^{-1}\boldsymbol{X}(\boldsymbol{X}^{\top}\boldsymbol{\Omega}^{-1}\boldsymbol{X}+\boldsymbol{K}_1)^{-1}\boldsymbol{X}^{\top}\}\boldsymbol{\Omega}^{-1}\{\boldsymbol{I}_n-\boldsymbol{X}(\boldsymbol{X}^{\top}\boldsymbol{\Omega}^{-1}\boldsymbol{X}+\boldsymbol{K}_1)^{-1}\boldsymbol{X}^{\top}\boldsymbol{\Omega}^{-1}\}\boldsymbol{Z}\notag\\
&=\{\boldsymbol{X}^{\top}-\boldsymbol{A}(\boldsymbol{A}+\boldsymbol{K}_1)^{-1}\boldsymbol{X}^{\top}\}\boldsymbol{\Omega}^{-1}\{\boldsymbol{Z}-\boldsymbol{X}(\boldsymbol{A}+\boldsymbol{K}_1)^{-1}\boldsymbol{B}\}\notag\\
&=\{\boldsymbol{I}_k-\boldsymbol{A}(\boldsymbol{A}+\boldsymbol{K}_1)^{-1}\}\boldsymbol{X}^{\top}\boldsymbol{\Omega}^{-1}\{\boldsymbol{Z}-\boldsymbol{X}(\boldsymbol{A}+\boldsymbol{K}_1)^{-1}\boldsymbol{B}\}\notag\\
&=\{\boldsymbol{I}_k-\boldsymbol{A}(\boldsymbol{A}+\boldsymbol{K}_1)^{-1}\}\{\boldsymbol{X}^{\top}\boldsymbol{\Omega}^{-1}\boldsymbol{Z}-\boldsymbol{X}^{\top}\boldsymbol{\Omega}^{-1}\boldsymbol{X}(\boldsymbol{A}+\boldsymbol{K}_1)^{-1}\boldsymbol{B}\}\notag\\
&=\{\boldsymbol{I}_k-\boldsymbol{A}(\boldsymbol{A}+\boldsymbol{K}_1)^{-1}\}\{\boldsymbol{B}-\boldsymbol{A}(\boldsymbol{A}+\boldsymbol{K}_1)^{-1}\boldsymbol{B}\}\notag\\
&=\{\boldsymbol{I}_k-\boldsymbol{A}(\boldsymbol{A}+\boldsymbol{K}_1)^{-1}\}\{\boldsymbol{I}_k-\boldsymbol{A}(\boldsymbol{A}+\boldsymbol{K}_1)^{-1}\}\boldsymbol{B} \notag\\
&=-\{\boldsymbol{I}_k-\boldsymbol{A}(\boldsymbol{A}+\boldsymbol{K}_1)^{-1}\}\{\boldsymbol{I}_k-\boldsymbol{A}(\boldsymbol{A}+\boldsymbol{K}_1)^{-1}\}\boldsymbol{A}\,\boldsymbol{\Xi} \boldsymbol{\Delta}^{-1}\notag\\
&=-\{\boldsymbol{I}_k-\boldsymbol{A}(\boldsymbol{A}+\boldsymbol{K}_1)^{-1}\}\{\boldsymbol{A}-\boldsymbol{A}(\boldsymbol{A}+\boldsymbol{K}_1)^{-1}\boldsymbol{A}\}\,\boldsymbol{\Xi} \boldsymbol{\Delta}^{-1}\notag\\
&=-\{\boldsymbol{I}_k-\boldsymbol{A}(\boldsymbol{A}+\boldsymbol{K}_1)^{-1}\}\boldsymbol{A}\{\boldsymbol{I}_k-(\boldsymbol{A}+\boldsymbol{K}_1)^{-1}\boldsymbol{A}\}\,\boldsymbol{\Xi} \boldsymbol{\Delta}^{-1},
\end{align}
and
\begin{align}
&\boldsymbol{Z}^{\top}\boldsymbol{S}^{\top}\boldsymbol{\Omega}^{-1}\boldsymbol{S}\boldsymbol{Z} \notag\\
&= \boldsymbol{Z}^{\top}\{\boldsymbol{I}_n-\boldsymbol{\Omega}^{-1}\boldsymbol{X}(\boldsymbol{X}^{\top}\boldsymbol{\Omega}^{-1}\boldsymbol{X}+\boldsymbol{K}_1)^{-1}\boldsymbol{X}^{\top}\}\boldsymbol{\Omega}^{-1}\{\boldsymbol{I}_n-\boldsymbol{X}(\boldsymbol{X}^{\top}\boldsymbol{\Omega}^{-1}\boldsymbol{X}+\boldsymbol{K}_1)^{-1}\boldsymbol{X}^{\top}\boldsymbol{\Omega}^{-1}\}\boldsymbol{Z}\notag\\
&=\{\boldsymbol{Z}^{\top}-\boldsymbol{C}(\boldsymbol{A}+\boldsymbol{K}_1)^{-1}\boldsymbol{X}^{\top}\}\boldsymbol{\Omega}^{-1}\{\boldsymbol{Z}-\boldsymbol{X}(\boldsymbol{A}+\boldsymbol{K}_1)^{-1}\boldsymbol{B}\}\notag\\ 
&=\boldsymbol{Z}^{\top}\boldsymbol{\Omega}^{-1}\boldsymbol{Z}-\boldsymbol{Z}^{\top}\boldsymbol{\Omega}^{-1}\boldsymbol{X}(\boldsymbol{A}+\boldsymbol{K}_1)^{-1}\boldsymbol{B}-\boldsymbol{C}(\boldsymbol{A}+\boldsymbol{K}_1)^{-1}\boldsymbol{X}^{\top}\boldsymbol{\Omega}^{-1}\boldsymbol{Z}\notag\\ 
&+\boldsymbol{C}(\boldsymbol{A}+\boldsymbol{K}_1)^{-1}\boldsymbol{X}^{\top}\boldsymbol{\Omega}^{-1}\boldsymbol{X}(\boldsymbol{A}+\boldsymbol{K}_1)^{-1}\boldsymbol{B}\notag\\ 
&= \boldsymbol{D}-\boldsymbol{C}(\boldsymbol{A}+\boldsymbol{K}_1)^{-1}\boldsymbol{B}-\boldsymbol{C}(\boldsymbol{A}+\boldsymbol{K}_1)^{-1}\boldsymbol{B}+\boldsymbol{C}(\boldsymbol{A}+\boldsymbol{K}_1)^{-1}\boldsymbol{A}(\boldsymbol{A}+\boldsymbol{K}_1)^{-1}\boldsymbol{B}\notag\\ 
&= \boldsymbol{\Delta}^{-1}+\boldsymbol{B}^{\top}\boldsymbol{A}^{-1}\boldsymbol{B}-2\boldsymbol{B}^{\top}(\boldsymbol{A}+\boldsymbol{K}_1)^{-1}\boldsymbol{B}+\boldsymbol{B}^{\top}(\boldsymbol{A}+\boldsymbol{K}_1)^{-1}\boldsymbol{A}(\boldsymbol{A}+\boldsymbol{K}_1)^{-1}\boldsymbol{B}\notag\\
&= \boldsymbol{\Delta}^{-1}+\boldsymbol{B}^{\top}\{\boldsymbol{A}^{-1}-2(\boldsymbol{A}+\boldsymbol{K}_1)^{-1}+(\boldsymbol{A}+\boldsymbol{K}_1)^{-1}\boldsymbol{A}(\boldsymbol{A}+\boldsymbol{K}_1)^{-1}\}\boldsymbol{B}\notag\\ 
&=  \boldsymbol{\Delta}^{-1}+\boldsymbol{B}^{\top}\{\boldsymbol{A}^{-1}-(\boldsymbol{A}+\boldsymbol{K}_1)^{-1}\}\boldsymbol{A}\{\boldsymbol{A}^{-1}-(\boldsymbol{A}+\boldsymbol{K}_1)^{-1}\}\boldsymbol{B}\notag\\ 
&=  \boldsymbol{\Delta}^{-1}+\boldsymbol{\Delta}^{-1}\boldsymbol{\Xi}^{\top}\boldsymbol{A}\{\boldsymbol{A}^{-1}-(\boldsymbol{A}+\boldsymbol{K}_1)^{-1}\}\boldsymbol{A}\{\boldsymbol{A}^{-1}-(\boldsymbol{A}+\boldsymbol{K}_1)^{-1}\}\boldsymbol{A}\boldsymbol{\Xi}\boldsymbol{\Delta}^{-1}\notag\\
&=  \boldsymbol{\Delta}^{-1}+\boldsymbol{\Delta}^{-1}\boldsymbol{\Xi}^{\top}\{\boldsymbol{I}_k - \boldsymbol{A}(\boldsymbol{A}+\boldsymbol{K}_1)^{-1}\}\boldsymbol{A}\{\boldsymbol{I}_k-(\boldsymbol{A}+\boldsymbol{K}_1)^{-1}\boldsymbol{A}\}\boldsymbol{\Xi}\boldsymbol{\Delta}^{-1}. \notag
\end{align}
A straightforward calculation yields that
\begin{align*}
\boldsymbol{T}\boldsymbol{T}= \boldsymbol{I}_n -2\boldsymbol{X}(\boldsymbol{X}^{\top}\boldsymbol{X}+\boldsymbol{K}_2)^{-1}\boldsymbol{X}^\top + \boldsymbol{X}(\boldsymbol{X}^{\top}\boldsymbol{X}+\boldsymbol{K}_2)^{-1}\boldsymbol{X}^\top \boldsymbol{X}(\boldsymbol{X}^{\top}\boldsymbol{X}+\boldsymbol{K}_2)^{-1}\boldsymbol{X}^\top.
\end{align*} 
The equality \eqref{XXK} can be rewritten as 
\begin{align}\label{XXK2}
\{\boldsymbol{I}_k-\boldsymbol{A}(\boldsymbol{A}+\boldsymbol{K}_1)^{-1}\}\boldsymbol{A}\{\boldsymbol{I}_k-(\boldsymbol{A}+\boldsymbol{K}_1)^{-1}\boldsymbol{A}\}=\boldsymbol{X}^{\top}\boldsymbol{T}\boldsymbol{T}\boldsymbol{X},    
\end{align}
and \eqref{XZK} is equivalent to 
\begin{align}\label{XZK2}
& -\{\boldsymbol{I}_k-\boldsymbol{A}(\boldsymbol{A}+\boldsymbol{K}_1)^{-1}\}\boldsymbol{A}\{\boldsymbol{I}_k-(\boldsymbol{A}+\boldsymbol{K}_1)^{-1}\boldsymbol{A}\}\boldsymbol{\Xi} \boldsymbol{\Delta}^{-1}=\boldsymbol{0} \notag\\
\Leftrightarrow \quad & \{\boldsymbol{I}_k-\boldsymbol{A}(\boldsymbol{A}+\boldsymbol{K}_1)^{-1}\}\boldsymbol{A}\{\boldsymbol{I}_k-(\boldsymbol{A}+\boldsymbol{K}_1)^{-1}\boldsymbol{A}\}\boldsymbol{\Xi}=\boldsymbol{0} \notag\\ 
\Leftrightarrow \quad & \boldsymbol{X}^{\top}\boldsymbol{T}\boldsymbol{T}\boldsymbol{X}\boldsymbol{\Xi}=\boldsymbol{0}
\end{align}
under \eqref{XXK2}.
Moreover, since the equality $\eqref{EQA}=\boldsymbol{0}$ holds, the equality \eqref{ZZK} is equivalent to
\begin{align*}
&\boldsymbol{\Delta}^{-1}+\boldsymbol{\Delta}^{-1}\boldsymbol{\Xi}^{\top}\{\boldsymbol{I}_k - \boldsymbol{A}(\boldsymbol{A}+\boldsymbol{K}_1)^{-1}\}\boldsymbol{A}\{\boldsymbol{I}_k-(\boldsymbol{A}+\boldsymbol{K}_1)^{-1}\boldsymbol{A}\}\boldsymbol{\Xi}\boldsymbol{\Delta}^{-1}=\boldsymbol{Z}^{\top}\boldsymbol{Z} \\
\Leftrightarrow \quad & \boldsymbol{\Delta}^{-1}=\boldsymbol{Z}^{\top}\boldsymbol{Z} \\
\Leftrightarrow \quad & \boldsymbol{\Delta}=(\boldsymbol{Z}^{\top}\boldsymbol{Z} )^{-1} .
\end{align*}
The equalities \eqref{XXK2} and \eqref{XZK2} can be rewritten as follows:
\begin{align*}
&\{\boldsymbol{I}_k-\boldsymbol{A}(\boldsymbol{A}+\boldsymbol{K}_1)^{-1}\}\boldsymbol{A}\{\boldsymbol{I}_k-(\boldsymbol{A}+\boldsymbol{K}_1)^{-1}\boldsymbol{A}\}=\boldsymbol{X}^{\top}\boldsymbol{T}\boldsymbol{T}\boldsymbol{X}\\   
\Leftrightarrow \quad &\{\boldsymbol{I}_k-\boldsymbol{A}(\boldsymbol{A}+\boldsymbol{K}_1)^{-1}\}\boldsymbol{A}\{\boldsymbol{I}_k-(\boldsymbol{A}+\boldsymbol{K}_1)^{-1}\boldsymbol{A}\}\\
&= \boldsymbol{X}^{\top}\{\boldsymbol{I}_n-\boldsymbol{X}(\boldsymbol{X}^{\top}\boldsymbol{X}+\boldsymbol{K}_2)^{-1}\boldsymbol{X}^{\top}\}\{\boldsymbol{I}_n-\boldsymbol{X}(\boldsymbol{X}^{\top}\boldsymbol{X}+\boldsymbol{K}_2)^{-1}\boldsymbol{X}^{\top}\}\boldsymbol{X}\\
\Leftrightarrow\quad &\{\boldsymbol{I}_k-\boldsymbol{A}(\boldsymbol{A}+\boldsymbol{K}_1)^{-1}\}\boldsymbol{A}\{\boldsymbol{I}_k-(\boldsymbol{A}+\boldsymbol{K}_1)^{-1}\boldsymbol{A}\}\\
&= \{\boldsymbol{X}^{\top}-\boldsymbol{X}^{\top}\boldsymbol{X}(\boldsymbol{X}^{\top}\boldsymbol{X}+\boldsymbol{K}_2)^{-1}\boldsymbol{X}^{\top}\}\{\boldsymbol{X}-\boldsymbol{X}(\boldsymbol{X}^{\top}\boldsymbol{X}+\boldsymbol{K}_2)^{-1}\boldsymbol{X}^{\top}\boldsymbol{X}\}\\
\Leftrightarrow\quad &\{\boldsymbol{I}_k-\boldsymbol{A}(\boldsymbol{A}+\boldsymbol{K}_1)^{-1}\}\boldsymbol{A}\{\boldsymbol{I}_k-(\boldsymbol{A}+\boldsymbol{K}_1)^{-1}\boldsymbol{A}\}\\
&= \{\boldsymbol{I}_k-\boldsymbol{X}^{\top}\boldsymbol{X}(\boldsymbol{X}^{\top}\boldsymbol{X}+\boldsymbol{K}_2)^{-1}\}\boldsymbol{X}^{\top}\boldsymbol{X}\{\boldsymbol{I}_k-(\boldsymbol{X}^{\top}\boldsymbol{X}+\boldsymbol{K}_2)^{-1}\boldsymbol{X}^{\top}\boldsymbol{X}\}\\
\Leftrightarrow\quad &\{\boldsymbol{I}_k-(\boldsymbol{A}+\boldsymbol{K}_1-\boldsymbol{K}_1)(\boldsymbol{A}+\boldsymbol{K}_1)^{-1}\}\boldsymbol{A}\{\boldsymbol{I}_k-(\boldsymbol{A}+\boldsymbol{K}_1)^{-1}(\boldsymbol{A}+\boldsymbol{K}_1-\boldsymbol{K}_1)\}\\
&=\{\boldsymbol{I}_k-(\boldsymbol{X}^{\top}\boldsymbol{X}+\boldsymbol{K}_2-\boldsymbol{K}_2)(\boldsymbol{X}^{\top}\boldsymbol{X}+\boldsymbol{K}_2)^{-1}\}\boldsymbol{X}^{\top}\boldsymbol{X}\{\boldsymbol{I}_k-(\boldsymbol{X}^{\top}\boldsymbol{X}+\boldsymbol{K}_2)^{-1}(\boldsymbol{X}^{\top}\boldsymbol{X}+\boldsymbol{K}_2-\boldsymbol{K}_2)\}\\
\Leftrightarrow\quad &\{\boldsymbol{I}_k-\boldsymbol{I}_k+\boldsymbol{K}_1(\boldsymbol{A}+\boldsymbol{K}_1)^{-1}\}\boldsymbol{A}\{\boldsymbol{I}_k-\boldsymbol{I}_k+(\boldsymbol{A}+\boldsymbol{K}_1)^{-1}\boldsymbol{K}_1\}\\
&= \{\boldsymbol{I}_k-\boldsymbol{I}_k+\boldsymbol{K}_2(\boldsymbol{X}^{\top}\boldsymbol{X}+\boldsymbol{K}_2)^{-1}\}\boldsymbol{X}^{\top}\boldsymbol{X}\{\boldsymbol{I}_k-\boldsymbol{I}_k+(\boldsymbol{X}^{\top}\boldsymbol{X}+\boldsymbol{K}_2)^{-1}\boldsymbol{K}_2\}\\
\Leftrightarrow\quad &\boldsymbol{K}_1(\boldsymbol{A}+\boldsymbol{K}_1)^{-1}\boldsymbol{A}(\boldsymbol{A}+\boldsymbol{K}_1)^{-1}\boldsymbol{K}_1=\boldsymbol{K}_2(\boldsymbol{X}^{\top}\boldsymbol{X}+\boldsymbol{K}_2)^{-1}\boldsymbol{X}^{\top}\boldsymbol{X}(\boldsymbol{X}^{\top}\boldsymbol{X}+\boldsymbol{K}_2)^{-1}\boldsymbol{K}_2,
\end{align*}
and 
\allowdisplaybreaks\begin{align*}
&\boldsymbol{X}^{\top}\boldsymbol{T}\boldsymbol{T}\boldsymbol{X}\boldsymbol{\Xi}=\boldsymbol{0}\\
\Leftrightarrow \quad &\{\boldsymbol{X}^{\top}\boldsymbol{X}-2\boldsymbol{X}^{\top}\boldsymbol{X}(\boldsymbol{X}^{\top}\boldsymbol{X}+\boldsymbol{K}_2)^{-1}\boldsymbol{X}^{\top}\boldsymbol{X}\\
&+\boldsymbol{X}^{\top}\boldsymbol{X}(\boldsymbol{X}^{\top}\boldsymbol{X}+\boldsymbol{K}_2)^{-1}\boldsymbol{X}^{\top}\boldsymbol{X}(\boldsymbol{X}^{\top}\boldsymbol{X}+\boldsymbol{K}_2)^{-1}\boldsymbol{X}^{\top}\boldsymbol{X}\}\boldsymbol{\Xi}=\boldsymbol{0}\\
\Leftrightarrow \quad &\{\boldsymbol{I}_k-2(\boldsymbol{X}^{\top}\boldsymbol{X}+\boldsymbol{K}_2)^{-1}\boldsymbol{X}^{\top}\boldsymbol{X}+(\boldsymbol{X}^{\top}\boldsymbol{X}+\boldsymbol{K}_2)^{-1}\boldsymbol{X}^{\top}\boldsymbol{X}(\boldsymbol{X}^{\top}\boldsymbol{X}+\boldsymbol{K}_2)^{-1}\boldsymbol{X}^{\top}\boldsymbol{X}\}\boldsymbol{\Xi}=\boldsymbol{0}\\
\Leftrightarrow \quad &\{(\boldsymbol{X}^{\top}\boldsymbol{X}+\boldsymbol{K}_2)-2\boldsymbol{X}^{\top}\boldsymbol{X}+\boldsymbol{X}^{\top}\boldsymbol{X}(\boldsymbol{X}^{\top}\boldsymbol{X}+\boldsymbol{K}_2)^{-1}\boldsymbol{X}^{\top}\boldsymbol{X}\}\boldsymbol{\Xi}=\boldsymbol{0}\\
\Leftrightarrow \quad &\{\boldsymbol{K}_2-\boldsymbol{X}^{\top}\boldsymbol{X}+(\boldsymbol{X}^{\top}\boldsymbol{X}+\boldsymbol{K}_2-\boldsymbol{K}_2)(\boldsymbol{X}^{\top}\boldsymbol{X}+\boldsymbol{K}_2)^{-1}\boldsymbol{X}^{\top}\boldsymbol{X}\}\boldsymbol{\Xi}=\boldsymbol{0}\\
\Leftrightarrow \quad &\{\boldsymbol{K}_2-\boldsymbol{X}^{\top}\boldsymbol{X}+\boldsymbol{X}^{\top}\boldsymbol{X}-\boldsymbol{K}_2(\boldsymbol{X}^{\top}\boldsymbol{X}+\boldsymbol{K}_2)^{-1}\boldsymbol{X}^{\top}\boldsymbol{X}\}\boldsymbol{\Xi}=\boldsymbol{0}\\
\Leftrightarrow \quad &\{\boldsymbol{K}_2-\boldsymbol{K}_2(\boldsymbol{X}^{\top}\boldsymbol{X}+\boldsymbol{K}_2)^{-1}\boldsymbol{X}^{\top}\boldsymbol{X}\}\boldsymbol{\Xi}=\boldsymbol{0}\\
\Leftrightarrow \quad &\boldsymbol{K}_2\{\boldsymbol{I}_k-(\boldsymbol{X}^{\top}\boldsymbol{X}+\boldsymbol{K}_2)^{-1}\boldsymbol{X}^{\top}\boldsymbol{X}\}\boldsymbol{\Xi}=\boldsymbol{0}\\
\Leftrightarrow \quad &\boldsymbol{K}_2\{\boldsymbol{I}_k-(\boldsymbol{X}^{\top}\boldsymbol{X}+\boldsymbol{K}_2)^{-1}(\boldsymbol{X}^{\top}\boldsymbol{X}+\boldsymbol{K}_2-\boldsymbol{K}_2)\}\boldsymbol{\Xi}=\boldsymbol{0}\\
\Leftrightarrow\quad &\boldsymbol{K}_2\{\boldsymbol{I}_k-\boldsymbol{I}_k+(\boldsymbol{X}^{\top}\boldsymbol{X}+\boldsymbol{K}_2)^{-1}\boldsymbol{K}_2\}\boldsymbol{\Xi}=\boldsymbol{0}\\
\Leftrightarrow \quad &\boldsymbol{K}_2(\boldsymbol{X}^{\top}\boldsymbol{X}+\boldsymbol{K}_2)^{-1}\boldsymbol{K}_2\boldsymbol{\Xi}=\boldsymbol{0}.
\end{align*}
This completes the proof.
\end{proof}

\begin{rem}
Let $\boldsymbol{K}_1 = \boldsymbol{K}_2 = \boldsymbol{0}$. 
Then, \eqref{Mcond21} and \eqref{Mcond22} are satisfied, so Theorem~\ref{thm2} is reduced to Proposition~\ref{prop1}.
\end{rem}

Hereafter, we consider the case of $\boldsymbol{K}_1,\boldsymbol{K}_2 \in \mathcal{S}^+(k)$ and proceed to the discussion of Theorem~\ref{thm2}.

\begin{cor}
For given  $\boldsymbol{K}_1, \boldsymbol{K}_2 \in \mathcal{S}^+(k)$, if \eqref{GRSE} holds, then 
$\hat{\boldsymbol{\beta}}(\boldsymbol{\Omega}, \boldsymbol{0}) = \hat{\boldsymbol{\beta}} (\boldsymbol{I}_n, \boldsymbol{0})$ and $RSS(\boldsymbol{\Omega},\boldsymbol{0}) = RSS(\boldsymbol{I}_n,\boldsymbol{0})$ both hold for any  $\boldsymbol{y} \in \mathbb{R}^n$.
\end{cor}

\begin{cor}\label{cor43}
Let $\boldsymbol{\Omega}$ be expressed in \eqref{Orep}.
For given $\boldsymbol{K} \in \mathcal{S}^+(k)$, $RSS(\boldsymbol{\Omega}, \boldsymbol{K}) = RSS(\boldsymbol{I}_n,\boldsymbol{K}) \ \mbox{for any} \ \boldsymbol{y} \in \mathbb{R}^n$  holds if and only if the following two conditions both hold:
\begin{align}
&  \boldsymbol{K}\{\boldsymbol{\Gamma}-(\boldsymbol{X}^{\top}\boldsymbol{X})^{-1}\}\boldsymbol{K}=\boldsymbol{X}^{\top}\boldsymbol{X}-\boldsymbol{\Gamma}^{-1}, \label{Cr431}\\
& \boldsymbol{\Omega}= \boldsymbol{X}\boldsymbol{\Gamma} \boldsymbol{X}^{\top}+\boldsymbol{Z}(\boldsymbol{Z}^{\top}\boldsymbol{Z})^{-1}\boldsymbol{Z}^{\top}\label{Cr432}.
\end{align}
\end{cor}
\begin{proof}
Letting $\boldsymbol{K}_1=\boldsymbol{K}_2=\boldsymbol{K}$, we will apply Theorem~\ref{thm2}. First, let us denote $\boldsymbol{A}=(\boldsymbol{\Gamma}-\boldsymbol{\Xi} \boldsymbol{\Delta}^{-1}\boldsymbol{\Xi}^\top)^{-1}$. 
It follows from \eqref{Mcond22} that $\boldsymbol{\Xi}=\boldsymbol{0}$. 
Additionally, $\boldsymbol{\Xi}=\boldsymbol{0}$ and \eqref{Mcond23} are equivalent to $\boldsymbol{\Omega}$ being of the form
\begin{align*}
\boldsymbol{\Omega}= \boldsymbol{X} \boldsymbol{\Gamma} \boldsymbol{X}^\top +\boldsymbol{Z} (\boldsymbol{Z}^\top \boldsymbol{Z})^{-1}\boldsymbol{Z}^\top \quad \mbox{for some} \ \ \boldsymbol{\Gamma} \in \mathcal{S}^+(k).
\end{align*}
Moreover, from $\boldsymbol{\Xi}=\boldsymbol{0}$, it holds that $\boldsymbol{A}=\boldsymbol{\Gamma}^{-1}$. So the equality \eqref{Mcond21} can be rewritten as
\begin{align*}
& \boldsymbol{K} (\boldsymbol{A} + \boldsymbol{K})^{-1} \boldsymbol{A} (\boldsymbol{A} + \boldsymbol{K})^{-1} \boldsymbol{K} = \boldsymbol{K}(\boldsymbol{X}^\top \boldsymbol{X} + \boldsymbol{K})^{-1} \boldsymbol{X}^\top \boldsymbol{X} (\boldsymbol{X}^\top \boldsymbol{X} + \boldsymbol{K})^{-1} \boldsymbol{K}\\
\Leftrightarrow \quad &  (\boldsymbol{\Gamma}^{-1}+ \boldsymbol{K})^{-1} \boldsymbol{\Gamma}^{-1} (\boldsymbol{\Gamma}^{-1} + \boldsymbol{K})^{-1}  = (\boldsymbol{X}^\top \boldsymbol{X} + \boldsymbol{K})^{-1} \boldsymbol{X}^\top \boldsymbol{X} (\boldsymbol{X}^\top \boldsymbol{X} + \boldsymbol{K})^{-1}\\
\Leftrightarrow \quad &(\boldsymbol{\Gamma}^{-1}+ \boldsymbol{K}) \boldsymbol{\Gamma} (\boldsymbol{\Gamma}^{-1} + \boldsymbol{K})  = (\boldsymbol{X}^\top \boldsymbol{X} + \boldsymbol{K})(\boldsymbol{X}^\top \boldsymbol{X})^{-1} (\boldsymbol{X}^\top \boldsymbol{X} + \boldsymbol{K}) \\
\Leftrightarrow \quad & \boldsymbol{\Gamma}^{-1} + 2\boldsymbol{K} + \boldsymbol{K}\boldsymbol{\Gamma} \boldsymbol{K} = \boldsymbol{X}^\top \boldsymbol{X} +2 \boldsymbol{K} + \boldsymbol{K}(\boldsymbol{X}^\top \boldsymbol{X} )^{-1}\boldsymbol{K}\\
\Leftrightarrow \quad & \boldsymbol{K}\{\boldsymbol{\Gamma}-(\boldsymbol{X}^{\top}\boldsymbol{X})^{-1}\}\boldsymbol{K}=\boldsymbol{X}^{\top}\boldsymbol{X}-\boldsymbol{\Gamma}^{-1}.
\end{align*}
This completes the proof.
\end{proof}
\begin{rem}
Let $\boldsymbol{\Omega}$ be expressed in \eqref{Orep}.
If $\boldsymbol{K}=\lambda \boldsymbol{I}_k \in \mathcal{S}^+(k)$, then Corollary~\ref{cor43} implies that $RSS(\boldsymbol{\Omega},\lambda \boldsymbol{I}_k)=RSS(\boldsymbol{I}_n,\lambda \boldsymbol{I}_k) \,\, \mbox{for any} \ \boldsymbol{y} \in \mathbb{R}^n$ holds if and only if the following two conditions both hold:
\begin{align*}
    & \lambda^2 \{\boldsymbol{\Gamma}-(\boldsymbol{X}^{\top}\boldsymbol{X})^{-1}\}=\boldsymbol{X}^{\top}\boldsymbol{X}-\boldsymbol{\Gamma}^{-1},\\
    & \boldsymbol{\Omega}= \boldsymbol{X}\boldsymbol{\Gamma} \boldsymbol{X}^{\top}+\boldsymbol{Z}(\boldsymbol{Z}^{\top}\boldsymbol{Z})^{-1}\boldsymbol{Z}^{\top}.
\end{align*}
See also Section~\ref{sec:5}.
\end{rem}

\section{Examples}\label{sec:5}
In this section, we demonstrate some concrete examples for which the equivalence conditions in Corollary~\ref{cormtpd} (i), (ii) and Corollary~\ref{cor43} hold. 
We consider a few simple cases for which $n=3$ and $k=2$.

\begin{exam}[example for Corollary~\ref{cormtpd} (i)]
Let
\begin{align*}
\boldsymbol{X}=
\begin{pmatrix}
1 & 0 \\
0 & 1 \\
0 & 0 
\end{pmatrix},
\
\boldsymbol{\Omega} =
\begin{pmatrix}
1 & 0 &0 \\
0 & \frac{1}{2} & 0 \\
0 & 0 & a 
\end{pmatrix}, \ 
\boldsymbol{K}_1=
\begin{pmatrix}
1 & 0 \\
0 & 2 
\end{pmatrix},
\ \mbox{and}\ \,
\boldsymbol{K}_2= \boldsymbol{I}_2,
\end{align*}
where $a$ is an arbitrary positive constant.
It follows from $\boldsymbol{K}^{-1}_2=\boldsymbol{I}_2$ that
\begin{eqnarray*}
    \boldsymbol{\Omega} \boldsymbol{X} \boldsymbol{K}^{-1}_2 \boldsymbol{K}_1&=&\begin{pmatrix}
 1 & 0 &0 \\
0 & \frac{1}{2} & 0 \\
0 & 0 & a 
\end{pmatrix}\begin{pmatrix}
1 & 0 \\
0 & 1 \\
0 & 0 
\end{pmatrix}
\begin{pmatrix}
1 & 0 \\
0 & 2 
\end{pmatrix}
\\
&=&\begin{pmatrix}
1 & 0  \\
0 & \frac{1}{2}  \\
0 & 0  
\end{pmatrix}\begin{pmatrix}
1 & 0 \\
0 & 2 
\end{pmatrix}\\
&=&\begin{pmatrix}
1 & 0 \\
0 & 1 \\
0 & 0 
\end{pmatrix}=\boldsymbol{X}.
\end{eqnarray*}
In other words, the equality 
\[ \hat{{\boldsymbol{\beta}}}(\boldsymbol{\Omega}, \boldsymbol{K}_1) = \hat{{\boldsymbol{\beta}}} (\boldsymbol{I}_3, \boldsymbol{K}_2)\left(=\cfrac{1}{2}
\begin{pmatrix}
1 & 0 & 0\\
0 & 1 & 0
\end{pmatrix}\boldsymbol{y}\right) \]
holds for any $\boldsymbol{y} \in \mathbb{R}^3$. 
\end{exam}

\begin{exam}[example for Corollary~\ref{cormtpd} (ii)]
Let
\begin{align*}
\boldsymbol{X}=
\begin{pmatrix}
1 & 2 \\
1 & 0 \\
0 & 1 
\end{pmatrix}, \
\boldsymbol{\Omega} =
\begin{pmatrix}
2 & -1 &-2 \\
-1 & 2 & 2 \\
-2 & 2 & 5 
\end{pmatrix},
\ \mbox{and} \ \,
\boldsymbol{K}=\lambda \boldsymbol{I}_2,
\end{align*}
where $\lambda$ is an arbitrary positive constant.
Then,
\begin{eqnarray*}
    \boldsymbol{\Omega} \boldsymbol{X} &=&\begin{pmatrix}
 2 & -1 &-2 \\
-1 & 2 & 2 \\
-2 & 2 & 5 
\end{pmatrix}\begin{pmatrix}
1 & 2 \\
1 & 0 \\
0 & 1 
\end{pmatrix}
\\
&=&\begin{pmatrix}
1 & 2 \\
1 & 0 \\
0 & 1 
\end{pmatrix}=\boldsymbol{X}.
\end{eqnarray*}
Therefore, the equality 
\[ \hat{{\boldsymbol{\beta}}}(\boldsymbol{\Omega}, \boldsymbol{K}) = \hat{{\boldsymbol{\beta}}} (\boldsymbol{I}_3, \boldsymbol{K})\left(=
\cfrac{1}{(\lambda+1)(\lambda+6)}
\begin{pmatrix}
\lambda+1 & \lambda+5 & -2\\
2(\lambda+1) & -2 & \lambda+2
\end{pmatrix}\boldsymbol{y}\right) \]
holds for any $\boldsymbol{y} \in \mathbb{R}^3$.
\end{exam}

\begin{exam}[example for Corollary~\ref{cor43}]
Let
\begin{align*}
\boldsymbol{X}=
\begin{pmatrix}
1 & 1 \\
1 & 0 \\
0 & 1 
\end{pmatrix}, \
\boldsymbol{Z}=
\begin{pmatrix}
 1 \\
-1  \\
-1 
\end{pmatrix}, \
\boldsymbol{\Omega} =\cfrac{1}{3}
\begin{pmatrix}
 3 & 0 & 0\\
0 & 2 & 1 \\
0 & 1 & 2 
\end{pmatrix},
\ \mbox{and} \ \,
\boldsymbol{K}=\sqrt{3} \boldsymbol{I}_2.
\end{align*}
If we suppose that \ $\boldsymbol{\Gamma} = 3^{-1}\boldsymbol{I}_2$, $\boldsymbol{\Xi}=\boldsymbol{0}$, and $\boldsymbol{\Delta}= 3^{-1}$, the covariance matrix $\boldsymbol{\Omega}$ can be written in the following form:
\begin{align*}
\boldsymbol{\Omega}
=\cfrac{1}{3}\begin{pmatrix}
1 & 1 \\
1 & 0 \\
0 & 1 
\end{pmatrix}
\boldsymbol{I}_2
\begin{pmatrix}
1 & 1 \\
1 & 0 \\
0 & 1
\end{pmatrix}^\top
+ \cfrac{1}{3} \begin{pmatrix}
 1 \\
-1  \\
-1 
\end{pmatrix}
\begin{pmatrix}
 1 \\
-1  \\
-1 
\end{pmatrix}^\top
=\boldsymbol{X}\boldsymbol{\Gamma} \boldsymbol{X}^{\top}+\boldsymbol{Z}(\boldsymbol{Z}^{\top}\boldsymbol{Z})^{-1}\boldsymbol{Z}^{\top},
\end{align*}
which implies that \eqref{Cr432}  is satisfied.\\
Additionally, for \eqref{Cr431}, the following equality
\begin{align*}
\boldsymbol{K}\{\boldsymbol{\Gamma}-(\boldsymbol{X}^{\top}\boldsymbol{X})^{-1}\}\boldsymbol{K}=\boldsymbol{X}^{\top}\boldsymbol{X}-\boldsymbol{\Gamma}^{-1}=
\begin{pmatrix}
-1 & 1 \\
1 & -1 
\end{pmatrix}
\end{align*}
holds. From the above, the equality
\[ RSS(\boldsymbol{\Omega}, \boldsymbol{K})=RSS(\boldsymbol{I}_3,\boldsymbol{K})\left(=\cfrac{1}{6}\,\boldsymbol{y}^{\top}
\begin{pmatrix}
2(3-\sqrt{3}) & -\sqrt{3} & -\sqrt{3}\\
-\sqrt{3}& 12-5\sqrt{3} & 2(3-\sqrt{3})\\
-\sqrt{3} & 2(3-\sqrt{3}) & 12-5\sqrt{3}
\end{pmatrix}\boldsymbol{y}\right) \]
holds for any $\boldsymbol{y} \in \mathbb{R}^3$.
\end{exam}

\section{Concluding remarks}\label{sec:6}
In this paper, we derived necessary and sufficient conditions under which two general ridge estimators coincide and equality between two residual sums of squares holds in a general linear model.
For future directions, it will be interesting to study the case where $\boldsymbol{\Omega}$ is not necessarily positive definite and $\boldsymbol{X}$ is not of full-rank.

\section*{Acknowledgments}
This study was supported in part by Japan Society for the Promotion of Science KAKENHI Grant Number 21K13836.

This is a pre-print of an article published in Statistical Papers.
The final authenticated version is available online at: \url{https://doi.org/10.1007/s00362-024-01644-z}.

\end{document}